\documentclass[a4paper,10pt]{article}
\usepackage{amssymb}
\usepackage{amsmath,amsfonts,amsthm,amssymb}
\usepackage[dvips]{graphics}
\usepackage{epsfig}
\usepackage{indentfirst}
\usepackage{color}
\usepackage{ulem}
\allowdisplaybreaks
\newtheoremstyle{theorem}%name
  {10pt}          % space above
  {10pt}  % space below
  {\sl}  % bofy font
 {}% ident - empty=no indent,  \parindent= paragraph indent
  {\bf}  % thm head font
  {. }    % punctuation after thm head
  { }    % space after thm head: `` ``=normal \newline=linebreak
  {}     % thm head specification
\theoremstyle{theorem}

\newtheorem{theorem}{Theorem}[section]
\newtheorem{corollary}{Corollary}[section]
\newtheorem{definition}{Definition}[section]
 \newtheorem{lemma}{Lemma}[section]
 \newtheorem{proposition}{Proposition}[section]
 \newtheorem{remark}{Remark}[section]
  
\numberwithin{equation}{section}

\newtheoremstyle{defi}%name
  {10pt}          % space above
  {10pt}  % space below
  {\rm}  % bofy font
  {}  % ident - empty=no indent,  \parindent= paragraph indent
  {\bf}  % thm head font
  {. }    % punctuation after thm head
  { }    % space after thm head: `` ``=normal \newline=linebreak
  {}     % thm head specification
\theoremstyle{defi}

\newcommand{\R}{\mathbb{R}}
\newcommand{\ep}{\varepsilon}
\newcommand{\de}{\partial}
\newcommand{\vr}{\rho}
\newcommand{\vu}{\mathbf{u}}
\newcommand{\vU}{\mathbf{U}}
\newcommand{\vm}{\mathbf{m}}
\newcommand{\vv}{\mathbf{v}}
\newcommand{\vn}{\mathbf{n}}
\newcommand{\vf}{\mathbf{f}}
\newcommand{\vrt}{\tilde{\rho}}
\newcommand{\vut}{\tilde{\vu}}

\renewcommand{\d}{\mathrm{d}}

\newcommand{\dy}{\,{\rm d} {y}}
\newcommand{\dt}{\,{\rm d} t }

\newcommand{\ds}{\,{\rm d} s }

\newcommand{\dydt}{\,{\rm d}y\,{\rm d}t}
\newcommand{\vphi}{{\boldsymbol \varphi}}
\newcommand{\vph}{{\boldsymbol \phi}}

\newcommand{\N}{\mathbb{N}}

\renewcommand{\S}{\mathbb{S}}

\newcommand{\supp}{{\rm supp}}

\newcommand{\Div}{{\rm div}_y}
\newcommand{\Grad}{\nabla_y}
\newcommand{\M}{\mathcal{M}}
\newcommand{\MM}{\mathcal{M}^+_{3 \times 3}}
\newcommand{\tr}{{\rm tr}}
\newcommand{\al}{\alpha}
\newcommand{\rhoin}{\rho_\infty}

\begin{document}
\baselineskip = 13.5pt

\title{On dissipative turbulent solutions to the compressible anisotropic Navier-Stokes equations in unbounded domains}

\author{Ond\v rej Kreml$^1$\footnote{Email: kreml@math.cas.cz} ,\ \ \
\v{S}\'{a}rka Ne\v{c}asov\'{a}$^1$\footnote{Email: matus@math.cas.cz} , \ \ \  Tong Tang$^{2}$\footnote{Email: tt0507010156@126.com}\\
{\small  1. Institute of Mathematics of the Czech Academy of Sciences,} \\
{\small \v Zitn\' a 25, 115 67, Praha 1, Czechia}\\
{\small 2. School of Mathematical Science,}\\
{\small Yangzhou University, Yangzhou 225002, P.R. China}\\
\date{}}
\maketitle
\begin{abstract}
%{\color{red}the work [Abbatiello et al. Generalized solutions to models of compressible viscous fluids, Discrete Contin. Dyn. Syst., 2021, 1-28]}, 
Inspired by Abbatiello, Feireisl and Novotn\' y \cite{e5}, we prove the global existence of dissipative turbulent solution for the compressible Navier-Stokes equations with anisotropic viscous stress tensor on unbounded domain. Our work complements the result of Bresch and Jabin \cite{br},
%{\color{red} [ Bresch and Jabin, Global existence of weak solutions for compressible Navier-Stokes equations: thermodynamically unstable pressure and anisotropic viscous stress tensor, Ann. of Math., 2018, 577-684]}
where the authors used the new compactness method to prove the existence of a weak solution to the same system in $\mathbb{T}^3$. By virtue of the concept of dissipative turbulent solutions, we are able to relax assumptions on the anisotropic tensor coefficients and the pressure law coefficient. We point out that we establish the existence result on a large class of unbounded domains, which is more conform to geophysical context. We also prove the weak-strong uniqueness property of acquired dissipative turbulent solutions.
\vspace{0.5cm}


\medskip

\end{abstract}
\maketitle
\section{Introduction}\label{s:1}\setcounter{equation}{0}
As is well known, the compressible Navier–Stokes equations form a vast subject with a long history of fruitful results in mathematics and various other scientific fields. A canonical assumption in the Navier–Stokes framework is that the viscous stress tensor is isotropic. However, with the rapid development of studies in geophysical flows, numerous works in meteorology and oceanography have shown that the vertical scale is much smaller than the horizontal scale. Consequently, a common physical assumption in geophysical flows is that the viscosity coefficient in the vertical direction is much smaller than that in the horizontal direction, rendering the viscous stress tensor anisotropic.

In what follows we denote the space variable by $y \in \R^3$ and use the notation $y = (x,z)$, where $x \in \R^2$ represents the horizontal variable and $z \in \R$ represents the vertical variable. We assume that the fluid occupies an unbounded domain $\Omega \subset \R^3$ and we assume that
\begin{equation}\label{eq:Omega}
    \Omega = \bigcup_{R = 1}^{\infty} \Omega_R, \quad \Omega_R \subset \Omega_{R+1},
\end{equation}
where $\Omega_R$ are bounded uniformly Lipschitz domains.

We consider the compressible anisotropic Navier-Stokes system in the space-time cylinder $(0,T)\times\Omega$:
\begin{equation}\label{1a}
\left\{
\begin{array}{llll}  \partial_{t}\rho+\Div(\rho \mathbf{u})=0, \\
\partial_t(\rho \mathbf{u})+\Div(\rho\mathbf{u}\otimes\mathbf{u})+\Grad p(\rho)=\mu\Delta_x\mathbf{u}+\delta\partial_{zz}\vu+(\mu+\lambda)\Grad\Div\mathbf{u}+\rho \vf,
\end{array}\right.
\end{equation}
where $\mu, \delta > 0$ and $\lambda \in \R$ are the viscosity coefficients, the unknown fields $\rho=\rho(t,y)$ and $\mathbf{u}=\mathbf{u}(t,y)$ represent the density and the velocity of the fluid, and $\vf$ stands for a given vector field representing the external volume force. The symbol $p(\rho)$ denotes the pressure. Throughout the paper we assume that the pressure takes the form
\begin{equation}\label{eq:pressurelaw}
    p(\rho) = \rho^{\gamma},\hspace{5pt}(\gamma > 1).
\end{equation}
%The symbol $\nabla=(\nabla_x,\partial_z)$, where the $x$ and $z$ represents the horizontal direction and vertical direction.
The system \eqref{1a} is supplemented with the no-slip boundary condition
\begin{equation}\label{eq:BC}
\mathbf u|_{\de\Omega}=0.
\end{equation}
Since we suppose the domain is unbounded, we state the far field conditions at infinity, that is
\begin{equation}\label{1b}
\lim_{|y|\rightarrow\infty}(\rho(t,y),\mathbf u(t,y))=(\rho_\infty,0),\hspace{3pt}\forall t>0,
\end{equation}
where $\rho_\infty > 0$ is a given constant. Initial conditions for system \eqref{1a} are given as
\begin{align}
    &\rho(0,y) = \rho_0(y), \nonumber \\
    &\rho\vu(0,y) = (\rho\vu)_0(y), \label{eq:IC}
\end{align}
and we assume that the initial conditions satisfy the compatibility condition
\begin{align*}
    \rho_0(y)  \rightarrow\rho_\infty,\hspace{6pt}
    (\rho\vu)_0(y)\rightarrow0,\hspace{5pt}\text{as}\hspace{3pt}|y|\rightarrow\infty.
\end{align*}

%The system is supplemented by the far field conditions boundary condition
%\begin{eqnarray}
%\mathbf{u}\cdot\mathbf{n}|_{\partial\Omega}=0,\label{1.3}
%\end{eqnarray}
%where $\mathbf{n}$ is outer normal vector to $\partial\Omega$.

The study of the well-posedness of the compressible Navier–Stokes equations is fundamental for understanding the mathematical structure of fluid dynamics. In 1980s, Matsumura and Nishida \cite{m} proved the global existence of classical solutions when the initial data are close to equilibrium. Then, Cho et al. \cite{c}, Salvi and Stra\v skraba \cite{st} obtained the local existence of strong solutions when the initial density contains vacuum.  Recently, Huang, Li and Xin \cite{xin} established the global existence of classical solutions to isotropic Navier-Stokes equations with small total energy but large oscillation. The major breakthrough in the theory of weak solutions was achieved by Lions \cite{l2} and Feireisl \cite{e, e1}, they proved the global existence of weak solutions with finite energy without any restriction on the size of initial data. The main obstacle in the existence of weak solutions is the compactness of the convective term $\rho\mathbf u \otimes \mathbf u$ and the pressure term $p(\rho)$. In order to deal with this issue and in particular to prove strong convergence of the sequence of densities, the effective viscous flux $F=(2\mu+\lambda)\Div\mathbf u-p(\rho)$ was used in an essential way in the theory of Lions and Feireisl. The crucial ingredient in the proof of global existence of finite energy weak solutions is the relation
\begin{equation}\label{1c}
\Delta_y F=\Div(\partial_t(\rho\mathbf u)+\Div(\rho\mathbf u\otimes\mathbf u)).
\end{equation}

Hoff and Smoller \cite{h}, and Serre \cite{s} used the effective viscous flux to study the propagation of oscillations and it is used as a key tool in the existence theory of various compressible fluid models, see \cite{duan,li1} as just a few examples.

We emphasize that the relation \eqref{1c} strongly depends on the isotropic form of the stress tensor. In geophysical flows, the stress tensor ceases to be isotropic because the horizontal and vertical directions exhibit distinct physical properties. If one tries to make the same steps with anisotropic stress tensor, the procedure leads to a non-local relation between $\Div\mathbf u$ and $p(\rho)$ in the effective viscous flux, as in this case, relation \eqref{1c} becomes
\begin{equation}
(\mu\Delta_x+\delta\partial_{zz}+(\mu+\lambda)\Delta_y)\Div\mathbf u-\Delta_y p(\rho)=\Div(\partial_t(\rho\mathbf u)+\Div(\rho\mathbf u\otimes\mathbf u)).
\end{equation}

As is pointed out by Bresch and Jabin \cite{br}, the anisotropic case seems to fall completely out the theory developed by Lions and Feireisl. Therefore, there are very few results about the existence of weak solutions for compressible anisotropic Navier-Stokes equations. The authors in \cite{br} introduced a new compactness method and used it to obtain global existence of weak solutions, which is the first global existence result for weak solutions of compressible anisotropic Navier-Stokes system. In order to achieve this groundbreaking result, the authors needed sufficient smallness of the difference $\mu-\delta$ and a rather high pressure law coefficient $\gamma > 2 + \frac{\sqrt{10}}{2}$ and they worked in the bounded domain $\mathbb{T}^3$. Later, based on the defect measures, Bresch and Burtea \cite{b1,b2} proved the global existence of weak solutions for the quasi-stationary compressible Stokes system. The same authors proved the existence of Hoff type weak solution for anisotropic stress tensor to Navier-Stokes equations, with large range for the coefficient $\gamma$ and initial conditions close to equilibrium, see \cite{b3}. Bocchi, Fanelli and Prange \cite{bo} used the maximal regularity to obtain local existence of strong solutions for the anisotropic Navier-Stokes equations.

Inspired by DiPerna's work \cite{d}, Feireisl et al. \cite{e5,j1,j2} proposed the concept of dissipative turbulent solution, i.e. solutions satisfying the equations and the energy inequality in the distributional sense but with extra defect measures. In the field of incompressible inviscid flows, Lions \cite{l1} used the terminology of dissipative solutions, this concept was then generalized by Brenier \cite{bre}. This notion of solution can be seen as a variant of a dissipative measure-valued solution, and can be constructed by convergence of certain numerical schemes. The key idea is that oscillation and concentration defects produced from convective term and pressure can be conveniently unified giving rise to a single positive definite Reynolds defect measure which is suitably related to the energy defect measure. Recently, the notion of dissipative turbulent solution is studied for various fluid models \cite{ba,ba1,e6} and its long time behavior is studied in \cite{e7}.

Regrading the fluid domain, it is reasonable to develop geophysical flow models and study existence of solutions in unbounded domains. The main strategy of constructing a weak solution in a given unbounded domain is to approximate the domain with a sequence of expanding bounded domains and show convergence of the sequence of solutions in appropriate function spaces, this is called the invading domain method. For incompressible Navier-Stokes equations, Heywood \cite{hey}, Robinson et al. \cite{r} (Chapter 4.4) used this method and proved the global existence of weak solutions in $\mathbb{R}^3$. Regarding compressible Navier-Stokes equations, Novotn\'{y} and his coauthors \cite{j,n} and Poul \cite{p} made significant contributions.

The objectives of this paper are multifold. First, it is to explore existence of dissipative turbulent solution on unbounded domains in general. The main reference \cite{e5} proves the existence on bounded domains and the extension to unbounded domains is not completely straightforward. As far as we know, the proof of existence of dissipative turbulent solutions on unbounded domains is not available in the literature at this point. It is notable that our proof naturally works also for isotropic Navier-Stokes system  on unbounded domain. Our second aim is to provide proof of existence of some type of global-in-time solutions to the anisotropic compressible Navier-Stokes system for reasonably large set of viscosity coefficients and without restrictions on the pressure law coefficent $\gamma$, therefore complement the work of Bresch and Jabin \cite{br} on weak solutions to this problem. The notion of dissipative turbulent solution is suitable for this task as it does not require proof of strong convergence of densities and is able to cover all $\gamma > 1$. Last but not least, in order to show that this class of solutions is relevant, we show the weak-strong uniqueness principle.

The paper is organized as follows. In Section \ref{s:2}, we recall some useful lemmas and introduce the definition of the dissipative turbulent solution. Our main results are stated in Section \ref{s:3}. Section \ref{s:4} is devoted to proof of existence of dissipative turbulent solution. In Section \ref{s:5} we prove that any dissipative turbulent solution satisfies the relative entropy inequality and we prove the weak-strong uniqueness property for dissipative turbulent solution in Section \ref{s:6}.

Throughout the paper we work with the assumption that the operator  $-\mu\Delta_x - \delta\de_{zz} - (\mu+\lambda)\Grad\Div$ is symmetric and strongly elliptic. This gives rise to a restriction on the viscosity coefficients identified in Lemma \ref{l:elip}. 

The first part of the proof of Theorem \ref{t3.1} closely follows the proof in \cite{e5}. However, we emphasize that the result of \cite{e5} cannot be used directly to claim the existence of dissipative turbulent solution on bounded domains $\Omega_R$, since the studied model \eqref{1a} does not fall into the general class of models studied in \cite{e5}. In particular the system \eqref{1a} arises from a non-symmetric stress tensor $\mathbb{S}$, whereas the authors in \cite{e5} assume symmetric stress tensors. This does not change the existence proof on bounded domains in a significant way, however we provide the details of the proof in sections \ref{ss:41}-\ref{s:limit_n} for completeness.

%%%%%%%%%%%%%%%%%%%%%%%%%%%%%%%%%%%%%%%%%%%%%%%%%%%%%%%%%%%%%%%%
%%%%%%%%%%%%%%%%%%%%%%%%%%%%%%%%%%%%%%%%%%%%%%%%%%%%%%%%%%%%%%%%
\section{Preliminaries}\label{s:2}

Before stating the definition of dissipative turbulent solution, we first introduce notation used in this work and some basic lemmas needed later.

We use the standard notation $L^p(M)$ and $W^{k,p}(M)$ for Lebesgue and Sobolev spaces on domain $M$ and we do not distinguish between spaces of scalar-, vector- and matrix-valued functions. $W^{k,p}_{0}(M)$ denotes the subspace of $W^{k,p}(M)$ consisting of functions with vanishing trace on $\de M$. The symbol $\M(\overline{M})$ stands for the space of (signed) Radon measures on $\overline{M}$, $\M^+(\overline{M})$ denotes the set of non-negative scalar Radon measures on $\overline{M}$ and $\M^+_{3\times 3}(\overline{M})$ stands for the set of symmetric positively semi-definite $3\times 3$ matrix-valued measures on $\overline{M}$. The trace of a matrix $\S$ is denoted by $\tr[\S]$.

The pressure potential is given by
\begin{equation}\label{eq:pPrel}
P'(\rho)\rho - P(\rho) = p(\rho),
\end{equation}
which yields
\begin{equation}\label{eq:PP}
    P(\rho) = \frac{\rho^{\gamma}}{\gamma-1}.
\end{equation}
Moreover, we define
\begin{equation}\label{eq:Einf}
\mathcal{E}^{\rho_\infty}(\rho) = P(\rho) - P(\rho_\infty) - P'(\rhoin)(\rho - \rhoin) = \frac{1}{\gamma-1}\left(\rho^\gamma+(\gamma-1)\rho^\gamma_\infty-\gamma\rho\rho^{\gamma-1}_\infty\right),
\end{equation}
which is a convex function.

The following lemma introduces necessary conditions for viscosity parameters $\mu, \delta$ and $\lambda$ used in this paper.

\begin{lemma}\label{l:elip}
    Let
    \begin{equation}\label{eq:visc_ass}
    \mu \geq \delta > 0 \quad \text{ and } \lambda > -\frac{\mu(\mu+3\delta)}{\mu+2\delta}.
    \end{equation}
    Then the operator $A = -\mu\Delta_x - \delta\de_{zz} - (\mu+\lambda)\Grad\Div$ is symmetric and strongly elliptic.
\end{lemma}
\begin{proof}
    The fact that $A$ is a symmetric operator is straightforward.   In order to prove strong ellipticity of the operator $A$ we identify $3\times 3$ matrices $\mathbb{A}^{jk} = \{a^{jk}_{lm}\}_{l,m=1}^3$ such that
    \begin{equation}
        (A \vu)_j = \sum_{k,l,m = 1}^3 a^{jk}_{lm}\frac{\de^2 u_k}{\de y_l \de y_m}
    \end{equation}
    and we need to prove that there exists a constant $\beta$ such that for any $\mathbb{B} = \{b_{lm}\}_{l,m=1}^3 \in \R^{3 \times 3}$ it holds
    \begin{equation}\label{eq:strong_el_def}
    \sum_{j,k,l,m = 1}^{3} a^{jk}_{lm} b_{jl} b_{km} \geq \beta|\mathbb{B}|^2.
    \end{equation}
    Straightforward calculations show that this is equivalent to
    \begin{equation}\label{eq:fb}
        f(\mathbb{B}) := \mu\left(\sum_{j=1}^3 b_{j1}^2 + b_{j2}^2\right) + \delta \sum_{j=1}^3 b_{j3}^2 + (\mu+\lambda)\left(b_{11} + b_{22} + b_{33}\right)^2 \geq \beta|\mathbb{B}|^2.
    \end{equation}
    Clearly both $\mu$ and $\delta$ need to be positive. Moreover, simplifying the notation to $b_j := b_{jj}$ for $j = 1,2,3$, $\lambda$ needs to be such that
    \begin{equation}
        (\mu-\beta) (b_1^2 + b_2^2) + (\delta-\beta) b_3^2 + (\mu+\lambda)(b_1+b_2+b_3)^2 \geq 0
    \end{equation}
    for some $\beta \in (0,\delta]$. As it is not difficult to show that
    \begin{equation}
        \inf_{b_1 + b_2 + b_3 \neq 0} \frac{b_1^2+b_2^2 + \frac{D}{M}b_3^2}{(b_1+b_2+b_3)^2} = \frac{D}{M+2D},
    \end{equation}
    we conclude that it needs to hold
    \begin{equation}\label{eq:VC0}
        V(\beta) := \frac{(\mu-\beta)(\delta-\beta)}{\mu+2\delta-3\beta} + \mu+\lambda \geq 0
    \end{equation}
    for some $\beta \in (0,\delta]$. It is a matter of simple computation to observe that $V(\beta)$ is a nonincreasing function of $\beta$, hence the necessary condition for \eqref{eq:VC0} to hold for some $\beta \in (0,\delta]$ is
    \begin{equation}
        V(0) = \frac{\mu\delta}{\mu+2\delta} + \mu+\lambda > 0,
    \end{equation}
    which is nothing else than \eqref{eq:visc_ass}.
    The constant $\beta$ in \eqref{eq:strong_el_def} is then given by
    \begin{equation}\label{eq:C0def}
        \beta = \min\left\{\delta, \frac{ 4\mu+\delta+3\lambda - \sqrt{12\mu^2+\delta^2+9\lambda^2-4\mu\delta+20\mu\lambda-2\lambda\delta}}{2}\right\}.
    \end{equation}
\end{proof}

\begin{corollary}\label{co:beta}
For $\mu, \delta, \lambda$ satisfying \eqref{eq:visc_ass} it holds
\begin{equation}
    \beta|\Grad \vu|^2 \leq \mu|\nabla_x \vu|^2 + \delta |\de_z \vu|^2 + (\mu+\lambda)|\Div \vu|^2
\end{equation}
with $\beta$ given by \eqref{eq:C0def}.
\end{corollary}

\begin{lemma}\label{l:convex}
   Let $\mu,\delta, \lambda$ satisfy \eqref{eq:visc_ass}. Then, the function $f : \R^{3\times3} \to \R$ defined in \eqref{eq:fb} is strictly convex.
\end{lemma}
\begin{proof}
    As $f(\mathbb{B})$ is a quadratic function of its nine entries, its Hessian matrix takes form
    \begin{equation*}
        \mathbb{H}_f = 2\left(
        \begin{array}{ccccccccc}
            2\mu+\lambda & 0 & 0 & 0 & \mu+\lambda & 0 & 0 & 0 & \mu+\lambda \\
            0 & \mu & 0 & 0 & 0 & 0 & 0 & 0 & 0 \\
            0 & 0 & \delta & 0 & 0 & 0 & 0 & 0 & 0 \\
            0 & 0 & 0 & \mu & 0 & 0 & 0 & 0 & 0 \\
            \mu + \lambda & 0 & 0 & 0 & 2\mu + \lambda & 0 & 0 & 0 & \mu + \lambda \\
            0 & 0 & 0 & 0 & 0 & \delta & 0 & 0 & 0 \\
            0 & 0 & 0 & 0 & 0 & 0 & \mu & 0 & 0 \\
            0 & 0 & 0 & 0 & 0 & 0 & 0 & \mu & 0 \\
            \mu + \lambda & 0 & 0 & 0 & \mu + \lambda & 0 & 0 & 0 & \mu+\lambda + \delta
        \end{array}
        \right).
    \end{equation*}
    Its eigenvalues are $\eta_1 = ... = \eta_5 = \mu$, $\eta_6 = \eta_7 = \delta$ and
    \begin{equation*}
        \eta_{8,9} = \frac{ 4\mu+\delta+3\lambda \pm \sqrt{12\mu^2+\delta^2+9\lambda^2-4\mu\delta+20\mu\lambda-2\lambda\delta}}{2}.
    \end{equation*}
    It follows from the proof of Lemma \ref{l:elip} (see \eqref{eq:C0def}) that $\eta_9 > \eta_8 > 0$. Hence, all eigenvalues of $\mathbb{H}_f$ are strictly positive, $\mathbb{H}_f$ is positive definite and $f$ is strictly convex.
\end{proof}
% {\color{red}
% \begin{lemma}(\cite{bo} Lemma 3.12)\label{anisotropic}
% Let $\Omega$ be $\mathbb{R}^3$. There exists a universal constant $C>0$ such that for all $\kappa>0$, and all $u\in H^1(\Omega)$, one has
% \begin{align}
% \|u\|_{L^6(\Omega)}\leq C\big{(}\kappa^{-\frac{1}{2}}\|\nabla_x u\|_{L^2(\Omega)}+\kappa\|\partial_z u\|_{L^2(\Omega)}\big{)}.
% \end{align}
% \end{lemma} }

% The following two lemmas show several properties of weakly continuous functions.

% \begin{lemma} (\cite{n} Lemma 6.2)
% Let $1<p,q\leq\infty$ and let $\Omega$ be a bounded Lipschitz domain of $\mathbb{R}^N$($N\geq2$). Let $\{g_n\}_{n\in \mathbb{N}}$ be sequence of functions defined on $\overline{I}$ with values in $L^q(\Omega)$ such that
% \begin{align*}
% g_n\in C^0(\overline{I};L^q_{weak}(\Omega)),\hspace{2pt}&g_n\hspace{2pt}\text{is uniformly continuous in}\hspace{2pt} W^{-1,p}(\Omega)\\
% & \text{and uniformly bounded in}\hspace{2pt} L^q(\Omega),
% \end{align*}
% Then there exists a subsequence such that
% \begin{align*}
% g_n\rightarrow g\hspace{2pt}\text{in}\hspace{2pt} C^0(\overline{I};L^q_{weak}(\Omega)).
% \end{align*}
% \end{lemma}

% \begin{lemma}(\cite{n} Lemma 6.4)
% Let $1<q\leq\infty$ and let $\Omega$ be a bounded Lipschitz domain of $\mathbb{R}^N$($N\geq2$). If $g_n\rightarrow g$ in $C^0(\overline{I};L^q_{weak}(\Omega))$, then
% $g_n\rightarrow g$ strongly in $L^p(I,W^{-1,r}(\Omega))$ for all $1\leq p<\infty$, and all $1\leq r<\frac{N}{N-1}$ or $\frac{N}{N-1}<r<\infty$ if $\frac{Nr}{N+r}<q<\infty$.

% \end{lemma}

Next lemma concerns the local weak compactness on unbounded domains, called invading domains lemma. In the case of $X = L^q_{loc}$ it was proved in \cite[Lemma 6.6]{n}, its generalization for Radon measures is straightforward. For convenience, we provide its proof.
\begin{lemma}\label{l2.3}
Let either $X = L^q_{loc}$, $1<q\leq\infty$, or $X = \mathcal{M}$. Let $M_0 \in \N$ and let $\{f_n\}$, $f_n\in L^p((0,T);X(\mathbb{R}^3))$, with $1< p \leq \infty$ be a sequence such that
\begin{align*}
\|f_n\|_{L^p(0,T;X(B_M))}\leq K(M) \hspace{2pt}\text{for}\hspace{2pt} M=M_0,M_0+1,M_0+2....
\end{align*}
Then there exists a subsequence such that $f_n\rightarrow f$ weakly$-\ast$ in $L^p(0,T;X(B_R))$ for any $R>0$.
\end{lemma}
\begin{proof}
    As a consequence of the Banach-Alaoglu theorem, there exists a subsequence $\{f_{n_k^0}\} \subset \{f_n\}$ such that $f_{n_k^0} \to f^0$ weakly$-\ast$ in $L^p(0,T;X(B_{M_0}))$. The same theorem implies also the existence of $\{f_{n_k^1}\} \subset \{f_{n_k^0}\}$ such that $f_{n_k^1} \to f^1$ weakly$-\ast$ in $L^p(0,T;X(B_{M_0 + 1}))$. Moreover it clearly holds $f^1 = f^0$ in $(0,T)\times B_{M_0}$. We continue by induction and construct a sequence of weak$-\ast$ limits $\{f^k\}$. We set
    \begin{equation}
        f(t,x) = f^k(t,x) \qquad \text{ for } x \in B_{M_0+k},
    \end{equation}
    which in the case $X = \mathcal{M}$ translates to
    \begin{equation}
        \langle f(t,\cdot),\varphi(t,\cdot) \rangle = \langle f^k(t,\cdot),\varphi(t,\cdot)\rangle \qquad \text{ for } \text{supp}\ \varphi(t,\cdot) \subset B_{M_0+k} \text{ for a.a. } t \in (0,T).
    \end{equation}
    Choosing a sequence $\{f_{n_j^j}\}$ we end up with $f_{n_j^j} \to f$ weakly$-\ast$ in $L^p(0,T,X(B_{M_0 + k}))$ for any $k \in \mathbb{N}$.
\end{proof}

The final lemma shows the lower weak semicontinuity for sequences of Radon measures.
\begin{lemma}(\cite{ev} Theorem 3.1 in Chapter 1)\label{l2.4}
Let $U$ be an open, bounded, smooth subset of $\mathbb{R}^n$. Assume $\mu_k\rightarrow\mu$ weakly in $\mathcal{M}(U)$, then for each open set $V\subset U$, we have
\begin{align}
\mu(V)\leq\liminf_{k\rightarrow\infty}\mu_k(V).
\end{align}

\end{lemma}

Next, we define the dissipative turbulent solution.

\begin{definition}\label{d:solution}

Assume that $\mathbf f\in L^\infty((0,T)\times\Omega)\cap L^\infty(0,T;L^1(\Omega)). $ A \emph{dissipative turbulent solution} of the anisotropic Navier-Stokes system \eqref{1a}--\eqref{eq:IC} on $(0,T)\times\Omega$ is a couple $(\rho,\vu)$ such that

\noindent
$\bullet$
\begin{align*}
&\rho\in C_{weak}([0,T];L^\gamma_{loc}(\Omega)),\hspace{3pt} \rho\geq0,\\
&\rho\mathbf u\in C_{weak}([0,T];L^{\frac{2\gamma}{\gamma+1}}_{loc}(\Omega)),\hspace{3pt}\mathbf u\in L^2(0,T;W^{1,2}_0(\Omega)), \\
&\rho|\mathbf{u}|^2\in L^\infty(0,T;L^1(\Omega)), \hspace{3pt} \mathcal{E}^{\rhoin}(\rho) \in L^\infty(0,T;L^1(\Omega)).
\end{align*}

\noindent
$\bullet$
%$\rho\in C_{weak}\big{(}(0,T);L^\gamma(K)\big{)}$ for any compact subset $K\subset\mathbb{R}^3$ and
The continuity equation is satisfied in the sense of distributions, i.e.
\begin{equation}\label{2.1}
\int_{\Omega}(\rho\varphi)(\tau,y) \dy - \int_{\Omega}\rho_0(y)\varphi(0,y) \dy = \int^\tau_0\int_{\Omega}\left(\rho\partial_t\varphi + \rho\mathbf{u}\cdot \Grad\varphi\right) \dydt,
\end{equation}
for all $\varphi\in C^1_c([0,T]\times\overline{\Omega})$ and almost all $\tau \in [0,T]$;

\noindent
$\bullet$
%$\rho\mathbf{u}\in C_{weak}\big{(}(0,T);L^\frac{2\gamma}{\gamma+1}\big{)}$ for any compact subset $K\subset\mathbb{R}^3$ and satisfy t
There exists a Reynolds defect measure $\mathfrak{R}\in L^\infty(0,T;\MM(\overline{\Omega}))$ such that the
momentum equation is satisfied in the sense
\begin{align}
&\int_{\Omega}(\rho\mathbf u\cdot\vphi)(\tau,y) \dy -
\int_{\Omega}(\rho\mathbf u)_0(y)\cdot\vphi(0,y) \dy - \int^\tau_0\int_\Omega\Grad\vphi:\d\mathfrak{R}(t)\dt \nonumber \\
&\qquad = \int^\tau_0\int_{\Omega}\left(\rho\mathbf u\cdot\partial_t\vphi + \rho\mathbf{u}\otimes\mathbf{u}:\Grad\vphi
 + p(\rho)\Div\vphi\right)\dydt \nonumber \\
&\qquad - \int^\tau_0\int_{\Omega} \left(\mu\nabla_x\mathbf u:\nabla_x\vphi + \delta\partial_z\mathbf u\cdot \partial_z\vphi + (\mu+\lambda)\Div\mathbf u\ \Div\vphi
+\rho \mathbf{f}\cdot\vphi\right)\dydt, \label{2.2}
\end{align}
for all $\vphi\in C^1_c([0,T]\times\overline{\Omega})$ such that $\vphi|_{\partial\Omega}=0$ and almost all $\tau \in [0,T]$;

\noindent
$\bullet$
There exists an energy defect measure $\mathfrak{E}\in L^\infty(0,T;\mathcal{M}^{+}(\overline{\Omega}))$ such that
the energy inequality holds as
\begin{equation}\label{2.3}
\begin{split}
&\int_{\Omega}\left(\frac{1}{2}\rho|\mathbf u|^2+\mathcal{E}^{\rho_\infty}(\rho)\right)(\tau,y)\dy +
\int_\Omega d\mathfrak{E}(\tau) \\
&\quad + \int^\tau_0\int_\Omega\left(\mu|\nabla_x\mathbf u|^2+\delta|\partial_z\mathbf u|^2+(\mu+\lambda)|\Div\mathbf u|^2\right)\dydt \\
&\quad \leq \int_{\Omega}\left(\frac{1}{2}\frac{|(\rho\mathbf u)_0|^2}{\rho_0}+\mathcal{E}^{\rho_\infty}(\rho_0)\right)(y)\dy + \int_0^\tau \int_\Omega \rho \vf\cdot\vu \dydt
\end{split}
\end{equation}
for almost all $\tau\in[0,T]$, where $\mathcal{E}^{\rhoin}$ is given by \eqref{eq:Einf};

\noindent
$\bullet$
The compatibility condition between the energy defect measure $\mathfrak{E}$ and the Reynolds defect $\mathfrak{R}$ holds, i.e. there exists constants $0<\underline{d}\leq\overline{d}$ such that
\begin{align}\label{2.5}
\underline{d}\mathfrak{E}\leq \tr[\mathfrak{R}]\leq \overline{d}\mathfrak{E}.
\end{align}

\end{definition}

%%%%%%%%%%%%%%%%%%%%%%%%%%%%%%%%%%%%%%%%%%%%%%%%%%%%%%%%%%%%%%%%%%%%%%%%%555
%%%%%%%%%%%%%%%%%%%%%%%%%%%%%%%%%%%%%%%%%%%%%%%%%%%%%%%%%%%%%%%%%%%%%%%%%%%%%%%%%%%%%
\section{Main result}\label{s:3}

Now, we are ready to state our main results. The first theorem is the existence of a dissipative turbulent solution.
\begin{theorem}\label{t3.1}
Let $\Omega\subset\mathbb{R}^3$ be an unbounded uniformly Lipschitz domain such that $\Omega=\bigcup^\infty_{R=1}\Omega_R$, where $\Omega_R$ are bounded Lipschitz domains that satisfy $\Omega_R\subset\Omega_{R+1}$. Let $T > 0$, $\gamma > 1$ and let $\mu, \delta$ and $\lambda$ satisfy \eqref{eq:visc_ass}. Let $\vf \in L^\infty((0,T)\times\Omega) \cap L^\infty(0,T;L^1(\Omega))$. Let $\rho_0 \in L^1_{loc}(\Omega)$, $\rho_0 \geq 0$, $(\rho\vu)_0 \in L^{\frac{2\gamma}{\gamma+1}}_{loc}(\Omega)$ such that $(\rho\vu)_0 = 0$ whenever $\rho_0 = 0$ and let
\begin{equation}
    \int_{\Omega} \left(\frac12 \frac{|(\rho\vu)_0|^2}{\rho_0} + \mathcal{E}^{\rho_\infty}(\rho_0) \right) \dy < \infty.
\end{equation}
Then the anisotropic compressible Navier-Stokes system \eqref{1a}--\eqref{eq:IC} admits at least one dissipative turbulent solution $(\rho,\vu)$ on $(0,T)\times\Omega$ in the sense of Definition \ref{d:solution}.
\end{theorem}

%\begin{remark}
%Here our pressure is the typical monotone case $p=\rho^\gamma$. A brief inspection of the proofs reveals that all principal results remain valid for the following case
%\begin{align*}
%p=\rho^\gamma+c\rho, \hspace{3pt}c>0.
%\end{align*}
%\end{remark}

% \begin{remark}
% We should emphasize that, when the domain is bounded or the viscosity coefficient is isotropic, our result is still valid with some modifications on the proof. Moreover, we will investigate another case
% \begin{align*}
% \lim_{|x|\rightarrow\infty}(\rho(x,t),\mathbf u(x,t))=(\overline{\rho},0),\hspace{3pt}\forall t>0,\hspace{2pt}\text{where}\hspace{2pt}\overline{\rho}\hspace{2pt}\text{is constant},
% \end{align*}
% in future study.
% \end{remark}

% \begin{remark}
% For bounded domain $\Omega$, the existence of dissipative turbulent solution has been proved in \cite{e5}. Theorem \ref{t3.1} can be viewed as a generalization of \cite{e5} to unbounded domains at the case of anisotropic.
% \end{remark}

\begin{theorem}\label{t:REI}
    Every dissipative turbulent solution to the system \eqref{1a}-\eqref{eq:IC} in the sense of Definition \ref{d:solution} satisfies the following relative entropy inequality
    \begin{equation}\label{eq:REI}
    \begin{split}
        &\int_{\Omega}\left(\frac 12 \rho |\vu-\vU|^2 + \mathcal{E}^r(\rho)\right)(\tau,y) \dy + \int_\Omega \d\mathfrak{E}(\tau) \\
        &\quad + \int^\tau_0\int_\Omega\left(\mu|\nabla_x\mathbf u|^2+\delta|\partial_z\mathbf u|^2+(\mu+\lambda)|\Div\mathbf u|^2\right)\dydt \\
        &\quad \leq \int_{\Omega}\left(\frac{1}{2}\frac{|(\rho\mathbf u)_0 - \rho_0\vU_0|^2}{\rho_0}+\mathcal{E}^{r(0,\cdot)}(\rho_0)\right)(y)\dy + \int_0^\tau \mathcal{R}(\rho,\vu,\mathfrak{R},r,\vU)(t) \dt
    \end{split}
    \end{equation}
    for almost all $\tau \in (0,T)$ and all test functions $r,\vU$ such that $r-\rhoin, \vU \in C^1_c([0,T]\times\overline{\Omega})$, $\vU|_{\partial\Omega} = 0$. Here, similarly as in \eqref{eq:Einf}, 
    \begin{equation}\label{eq:Erdef}
        \mathcal{E}^r(\rho) = P(\rho) - P(r) - P'(r)(\rho-r) = \frac{1}{\gamma-1}\left(\rho^\gamma+(\gamma-1)r^\gamma-\gamma\rho r^{\gamma-1}\right)
    \end{equation}
    and the remainder term is given as
    \begin{equation}\label{eq:Remainder}
    \begin{split}
    &\mathcal{R}(\rho,\vu,\mathfrak{R},r,\vU)(t) = \int_\Omega \left(\mu\nabla_x\mathbf u:\nabla_x\vU+\delta\partial_z\mathbf u\cdot\partial_z\vU+(\mu+\lambda)\Div\mathbf u\ \Div\vU\right)(t,y)\dy \\
    &\quad - \int_\Omega\left( \rho(\vu-\vU) \cdot(\partial_t\vU + \vu\cdot\nabla_y\vU) + p(\rho)\Div\vU + (\rho-r) \partial_t P'(r)\right)(t,y) \dy \\
    &\quad - \int_\Omega \left( \rho\vu\cdot\nabla_yP'(r) - \rho\vf\cdot(\vu-\vU)\right)(t,y) \dy - \int_\Omega \nabla_y\vU: \d\mathfrak{R}(t).
    \end{split}
    \end{equation}
\end{theorem}

The following theorem is the weak-strong uniqueness property in the class of dissipative turbulent solution for unbounded domain.
\begin{theorem}\label{t:WSU}
Let $\Omega\subset\mathbb{R}^3$ be an unbounded uniformly Lipschitz domain such that $\Omega=\bigcup^\infty_{R=1}\Omega_R$, where $\Omega_R$ are bounded Lipschitz domains that satisfy $\Omega_R\subset\Omega_{R+1}$. Let $T > 0$, $\gamma > 1$ and let $\mu, \delta$ and $\lambda$ satisfy \eqref{eq:visc_ass} and let $\vf = 0$.
Let $(\rho,\mathbf u)$ be a dissipative turbulent solution to the system \eqref{1a}-\eqref{eq:IC} with the energy defect measure $\mathfrak{E}$ and the Reynolds defect measure $\mathfrak{R}$ and let $(\vrt,\vut)$ be a strong solution of the same problem belonging to the class
\begin{equation}\label{eq:WSU_reg}
\begin{split}
&0 < \inf_{(0,T)\times\Omega}\vrt \leq \vrt(t,y) \leq \sup_{(0,T)\times\Omega}\vrt < \infty, \\
&\nabla_y\vrt \in L^1(0,T;L^{\frac{2\gamma}{\gamma-1}}(\Omega)),\\
&\nabla_y \vut \in L^1(0,T;L^\infty(\Omega)), \\ 
&\nabla^2_y \vut \in L^2(0,T;L^3(\Omega)) \cap L^1(0,T;L^\infty(\Omega)).
\end{split}
\end{equation}
Then
\begin{equation*}
\rho=\vrt,\hspace{2pt}\mathbf u=\vut \quad \text{ in }  (0,T)\times\Omega, \quad \text{ and } \mathfrak{E}=\mathfrak{R}=0.
\end{equation*}
\end{theorem}

Here we point out that the local existence of strong solutions has been recently proved by Bocchi, Fanelli and Prange \cite{bo}, however in a bit different regularity class.
\begin{theorem}\label{t:BFP}
    Let $\gamma \geq 1$ and let $\mu \geq \delta > 0$, $\lambda > -\mu-\delta$, $\rhoin > 0$ and $\vf = 0$. Then there exists $\eta > 0$ sufficiently small such that for any initial data $\rho_0$ and $(\rho\vu)_0 = \rho_0\vu_0$ with $\rho_0-\rhoin \in H^2$, $\vu_0 \in B^{\frac 32}_{2,\frac 43}$ satisfying $\|\rho_0-\rhoin\|_{L^\infty} \leq \eta$ there exists a time $T^*$ and a unique strong solution $(\tilde \rho,\tilde \vu)$ to \eqref{1a}-\eqref{eq:IC} such that 
    \begin{equation}\label{eq:SSregularity}
    \begin{split}
     &\tilde \rho-\rhoin\in C([0,T^*]; H^2) \text{ with } \|\tilde\rho -\rhoin\|_{L^\infty(0,T^*;L^\infty)} \leq 4\eta, \\
     &\tilde \vu\in L^\infty(0,T^*;L^2)\cap L^2(0,T^*; L^\infty), \nabla_y\tilde\vu\in L^4(0,T^*;L^2)\cap L^2(0,T^*;L^\infty), \\
     &\nabla^2_y\tilde\vu \in L^4(0,T^*;L^2), \nabla^3_y\tilde\vu \in L^{\frac{4}{3}}(0,T^*;L^2), \partial_t\tilde\vu \in L^{\frac{4}{3}}(0,T^*;H^1).   
    \end{split}
    \end{equation}
    The solution moreover satisfies the energy inequality \eqref{2.3} with $\mathfrak{E} = 0$.
\end{theorem}
\begin{proof}
    See \cite[Theorem 1]{bo}. The authors prove the theorem in the whole space but state in the introduction that domains with boundaries may be handled by the same method at the price of more technical difficulties. They also do not state explicitly the condition for $\lambda$, however from their proof it is clear that $\lambda > -\mu-\delta$ needs to be satisfied. Note that this is less restrictive condition than \eqref{eq:visc_ass}. Finally, we note that Theorem \ref{t:BFP} is a special version of \cite[Theorem 1]{bo} applied to constant behaviour of density at infinity which then implies the choice $\vf = 0$.
\end{proof}

% \begin{remark}
% Recently,  used the maximal regularity to obtain the local existence of strong solutions for the anisotropic Navier-Stokes equations. The regularity of the solution constructed in \cite{bo} is as follows:
% \begin{align*}
% &r\in C([0,T]; H^2(\Omega)),\hspace{3pt} \mathbf U\in L^\infty(0,T;L^2(\Omega))\cap L^2(0,T; L^\infty(\Omega)), \\
% &\nabla\mathbf U\in L^4(0,T;L^2(\Omega))\cap L^2(0,T;L^\infty(\Omega)),\hspace{3pt}\nabla^2\mathbf U \in L^4(0,T;L^2(\Omega)),\\
% &\nabla^3\mathbf U\in L^{\frac{4}{3}}(0,T;L^2(\Omega)),\hspace{3pt} \partial_t\mathbf U\in L^{\frac{4}{3}}(0,T;H^1(\Omega))
% \end{align*}

% \end{remark}

\begin{remark}
It is important to point out that Feireisl and Novotn$\acute{y}$ \cite{e2} recently proved the weak-strong uniqueness for compressible Navier-Stokes system with isotropic viscosity coefficient at unbounded domain. And they gave three types results which include the initial data containing vacuum.
\end{remark}

The rest of the paper is devoted to the proof of Theorems \ref{t3.1}, \ref{t:REI} and \ref{t:WSU}.

%%%%%%%%%%%%%%%%%%%%%%%%%%%%%%%%%%%%%%%%%%%%%%%%%%%%%%%%%%%%%%%%%%%%%%%%%%
%%%%%%%%%%%%%%%%%%%%%%%%%%%%%%%%%%%%%%%%%%%%%%%%%%%%%%%%%%%%%%%%%%%%%%%%%55

\section{Proof of Theorem \ref{t3.1}}\label{s:4}
The existence of dissipative turbulent solution for compressible systems arising from symmetric stress tensors on bounded domains $\Omega_R$ has been proved in \cite{e5}. We point out that the system \eqref{1a} does not fall into the general scheme presented in \cite{e5}, since it does not arise from a symmetric stress tensor. Therefore, for the sake of completeness, we will follow the main steps of the proof in  \cite{e5} and then combine them with the invading domains construction.

We introduce three parameters, $\ep > 0$ denoting the artificial diffusion in the continuity equation, $R \in \N$ denoting the domain $\Omega_R$ where we solve the approximated problem and $n \in \N$ denoting the finite dimension of the function space for the velocity field.

The anisotropic Lam\'e operator $A = -\mu\Delta_x-\delta\partial_{zz}-(\mu+\lambda)\Grad\Div$ enjoys the symmetric and strongly elliptic property as was proved in Lemma \ref{l:elip}. We consider $A$ as an operator from $L^2(\Omega_R)$ to $L^2(\Omega_R)$ with domain $\mathcal{D}(A) = W^{1,2}_0(\Omega_R) \cap W^{2,2}(\Omega_R)$. Standard theory of elliptic operators implies (see also \cite[Lemmas 4.32, 4.33]{n}) that there exists countable sets
\begin{align}
&\{\lambda_i^R\}^{\infty}_{i=1},\hspace{5pt}0<\lambda_1^R\leq\lambda_2^R\leq...,\\
&\{\vph_i^R\}^{\infty}_{i=1}\in W_0^{1,2}(\Omega_R)\cap W^{2,2}(\Omega_R),%,\hspace{3pt}1\leq p<\infty,
\end{align}
such that
\begin{equation}
-\mu\Delta_x\vph_i^R-\delta\partial_{zz}\vph_i^R-(\mu+\lambda)\Grad\Div\vph_i^R=\lambda_i^R\vph_i^R
\end{equation}
and $\{\vph_i^R\}$ is an orthonormal basis in $L^2(\Omega_R)$ and an orthogonal basis in $W_0^{1,2}(\Omega_R)$ with respect to the scalar product $\int_{\Omega_R}(\mu\nabla_x\vu:\nabla_x\vv + \delta\de_z\vu\cdot\de_z\vv + (\mu+\lambda)\Div\vu\ \Div \vv) \dy$. We denote $X_n^R={\rm span}\{\vph_i^R\}^n_{i=1}$.

Next, we introduce the initial conditions for the approximated problem. Given a function $\rho_0 \in L^1(\Omega) \cap L^\gamma(\Omega)$ with $\rho_0 \geq 0$, we construct a sequence $\rho_{0,n} \in C^1(\overline{\Omega})$ such that
\begin{equation}\label{eq:approxICprop}
\overline{\rho_n} > \rho_{0,n} \geq \underline{\rho_{n}} > 0
\end{equation}
for $\overline{\rho_n}, \underline{\rho_n} \in \R^+$ and $\rho_{0,n} \to \rho_0$ in $L^\gamma(\Omega)$. Then, $\rho_{0,n,R}$ is defined as a restriction of $\rho_{0,n}$ to the domain $\Omega_R$ and similarly we denote $\rho_{0,R}$ the restriction of $\rho_0$ to $\Omega_R$.

The initial condition for the velocity is given by $\vu_{0,n,R} \in X_n^R$ such that $\rho_{0,n,R}\vu_{0,n,R} \to (\rho\vu)_{0,R}$ in $L^\frac{2\gamma}{\gamma+1}(\Omega_R)$ as $n \to \infty$, where $(\rho\vu)_{0,R}$ is the restriction of $(\rho\vu)_0$ to $\Omega_R$. Moreover we require $\rho_{0,n,R}|\vu_{0,n,R}|^2 \to \frac{|(\rho\vu)_{0,R}|^2}{\vr_{0,R}}$ in $L^1(\Omega_R)$.

We denote $\alpha = (\ep,n,R)$ to shorten the notation. The approximation problem itself consists of the following equations
\begin{equation}
\left\{
\begin{array}{llll}  \partial_{t}\rho_{\alpha}+\Div(\rho_{\al}\mathbf{u}_{\al})=\varepsilon\Delta_y\rho_{\al}, \\
\partial_t(\rho_{\al} \mathbf{u}_{\al} )+\Div(\rho_{\al} \mathbf{u}_{\al} \otimes\mathbf{u}_{\al} )+\Grad p(\rho_{\al} ) \\
\qquad =\mu\Delta_x\mathbf{u}_{\al}+\delta\partial_{zz}\vu_{\al}+(\mu+\lambda)\Grad\Div\mathbf{u}_{\al} -\varepsilon(\nabla_y\rho_{\al}\cdot \nabla_y)\mathbf{u}_{\al} + \rho_\al\vf,\label{4.1}
\end{array}\right.
\end{equation}
on $(0,T)\times\Omega_R$ coupled with boundary conditions
\begin{align}
    \left.\frac{\de \vr_\al}{\de \vn}\right|_{\de \Omega_R} &= 0, \label{eq:approxBC1} \\
    \vu_\al|_{\de \Omega_R} &= 0, \label{eq:approxBC2}
\end{align}
and initial conditions
\begin{align}
    \rho_{\alpha}(0,y) &= \rho_{0,n,R}(y), \label{eq:approxIC1} \\
    \vu_{\alpha}(0,y) &= \vu_{0,n,R}(y), \label{eq:approxIC2}
\end{align}
in $\Omega_R$.

Following \cite{e1}, we solve the approximation system \eqref{4.1}--\eqref{eq:approxIC2} step by step.

%%%%%%%%%%%%%%%%%%%%%%%%%%%%%%%%%%%%%%%%%%%%%%%%
\subsection{Continuity equation}\label{ss:41}
 For a given $\vu_\al \in C([0,T]; X_n^R)$, we consider the equation:
\begin{equation}
\partial_{t}\rho_\al +\Div (\rho_\al \mathbf{u}_\al )=\varepsilon\Delta_y\rho_\al \label{approximation continuity equation}
\end{equation}
with boundary condition \eqref{eq:approxBC1} and initial condition \eqref{eq:approxIC1} bounded from above and below as in \eqref{eq:approxICprop}. As a weak solution to this problem we consider function $\vr_\al$ satisfying
\begin{equation}\label{eq:CEweak}
\int_{\Omega_R}(\rho_\al\varphi)(\tau,y) \dy - \int_{\Omega_R}\rho_{0,n,R}(y)\varphi(0,y) \dy = \int^\tau_0\int_{\Omega_R}\left(\rho_\al\partial_t\varphi + (\rho_\al\vu_\al - \ep\Grad\rho_\al)\cdot\Grad\varphi\right) \dydt,
\end{equation}
for all $\varphi\in C^1([0,T]\times\overline{\Omega_R})$ and almost all $\tau \in [0,T]$.

The existence theory for this problem is well known, see for example \cite{n}, and we summarize it in the following proposition.

% For any fixed $\varepsilon>0$, $R \in \mathbb{N}$, $n\geq R$ and given $\textbf{u}_{\varepsilon,n} \in C([0,T]; X_n)$, let us focus on finding that unique weak  solution $\varrho_{\varepsilon,n} = \varrho[\textbf{u}_{\varepsilon,n}]$ of equation \eqref{approximation continuity equation}.
% According to \cite{n} Lemma 4.28 that there exists countable sets
% \begin{align}
% &\{\lambda_i\}^{\infty}_{i=1},\hspace{5pt}0<\lambda_1\leq\lambda_2\leq...,\\
% &\{\Phi_i\}^{\infty}_{i=1}\in W^{2,p},\hspace{3pt}1\leq p<\infty,
% \end{align}
% such that
% $\Delta\Phi_i=\lambda_i\Phi_i$, where $\{\Phi_i\}$ is an orthonormal basis in $L^2(\Omega_R)$ and an orthogonal basis in $W^{1,2}(\Omega_R)$. Denote $X_n=span\{\Phi_i\}$, $\Pi_n$ the orthogonal projection of $L^2(\Omega_R)$ onto $X_n$. Compared with system (1.1), there is a dissipation term $\varepsilon\Delta\rho$ on the right side of the continuity equation and $\mathbf u_{\varepsilon,n}$ is given, the originally hyperbolic equation becomes a parabolic equation. In order to solve the system, we need several properties of this equation as the following

\begin{proposition} \label{existence approximated densities}
Let $\varepsilon>0$, $R, n \in \mathbb{N}$ be fixed. For any given $\vu_\al \in C([0,T]; X_n^R)$, there exists a unique weak solution
\begin{equation*}
	\rho_{\al} \in L^\infty(0,T; W^{1,2}(\Omega_R)) \cap L^{2}(0,T;W^{2,2}(\Omega_R)), \quad \partial_t\rho_{\al} \in L^2(0,T;L^2(\Omega_R))
\end{equation*}
of equation \eqref{approximation continuity equation} with boundary condition \eqref{eq:approxBC1} and initial condition \eqref{eq:approxIC1} satisfying \eqref{eq:approxICprop}. Moreover,
\begin{itemize}
\item[(i)] The solution $\rho_{\al}$ satisfies
	\begin{equation} \label{bound above density}
	\| \rho_{\al} \|_{L^{\infty}((0,\tau) \times \Omega_R)} \leq \overline{\rho_n} \exp \left( \tau \|\Div \vu_{\al}\|_{L^{\infty}((0,T) \times \Omega_R)} \right),
	\end{equation}
 and
	\begin{equation} \label{bound below density}
	\inf_{(0,\tau) \times \Omega_R} \rho_{\al}(t,x)\geq \underline{\rho_n} \exp \left( -\tau \|\Div \vu_{\al}\|_{L^{\infty}((0,T) \times \Omega_R)}\right),
	\end{equation}
for any $\tau \in [0,T]$;% with
%   \begin{equation} \label{maximum approximated initial density}
%			\overline{\rho}_n:= \max_{\Omega_R}  \varrho_{0,n};
%	\end{equation}
% for any $\tau \in [0,T]$, with
% 	\begin{equation} \label{minimum approximated initial density}
% 	\underline{\rho}_n:= \min_{\Omega_R}  \rho_{0,n};
% 	\end{equation}
\item[(ii)] For $\vu_1, \vu_2  \in C([0,T]; X_n^R)$ satisfying
	\begin{equation*}
	\max_{i=1,2} \| \vu_i\|_{L^{\infty}(0,T; W^{1,\infty}(\Omega_R))} \leq C
	\end{equation*}
with some $C > 0$ let $\rho_i= \rho[\vu_i]$, $i=1,2$ be the solutions of the approximated continuity equation \eqref{approximation continuity equation} with boundary conditions \eqref{eq:approxBC1} and the same initial data \eqref{eq:approxIC1}. Then, for any $\tau \in [0,T]$ it holds
	\begin{equation} \label{difference two solution approximates densities}
	\| (\rho_1- \rho_2)(\tau, \cdot)\|_{L^2(\Omega_R)} \leq c_1 \| \vu_1 - \vu_2 \|_{L^{\infty}(0,\tau; W^{1,\infty}(\Omega_R))}
	\end{equation}
with $c_1=c_1(\varepsilon, \rho_{0,n,R}, T, C)$;
\item[(iii)] Let $\rho_{0,n,R} \in W^{1,\infty}(\Omega_R)$. Then it holds
    \begin{equation}\label{3.10}
	\varepsilon\|\Grad\rho\|^2_{L^\infty(0,T,L^2(\Omega_R))} \leq C
	\end{equation}
with $C=C(\rho_{0,n}, T)$.
\end{itemize}
\end{proposition}

% \begin{proof}
% The proof of $(i)-(iii)$ can seen in \cite{n}. In order to obtain the estimate \eqref{3.10}, we multiply the equation \eqref{approximation continuity equation} on $\partial_t\rho$ as
% \begin{align*}
% \int_{\Omega_R}|\partial_t\rho_{\varepsilon,n}|^2+\text{div}(\rho_{\varepsilon,n}\mathbf u_{\varepsilon,n})\partial_t\rho_{\varepsilon,n}dx
% =\varepsilon\int_{\Omega_R}\partial_t\rho_{\varepsilon,n}\Delta\rho_{\varepsilon,n}dx=-\frac{\varepsilon}{2}\frac{d}{dt}\int_{\Omega_R}|\nabla\rho_{\varepsilon,n}|^2.
% \end{align*}
% Then, we get \eqref{3.10} by integrating the above equation over time.

% \end{proof}

%%%%%%%%%%%%%%%%%%%%%%%%%%%%%%%%%%%%%
\subsection{Momentum equation}\label{ss:42}
We consider the weak formulation of the approximated momentum equation
\begin{align}\label{eq:MEweak}
&\int_{\Omega_R}(\rho_{\al} \vu_{\al}(\tau,y) \cdot\vphi)\dy - \int_{\Omega_R}(\rho_{0,n,R} \vu_{0,n,R})(y) \cdot\vphi \dy \nonumber \\
& \quad = \int^\tau_0\int_{\Omega_R}\left(\rho_{\al} \vu_{\al} \otimes\vu_{\al} :\Grad\vphi
+p(\rho_{\al} )\Div\vphi\right)\dydt \nonumber \\
&-\int^\tau_0\int_{\Omega_R}\left(\mu\nabla_x \vu_{\al} :\nabla_x\vphi + \delta\partial_z\vu_{\al} \cdot \partial_z\vphi
+(\mu+\lambda)\Div\vu_{\al}\, \Div\vphi\right) \dydt\nonumber\\
&-\int^\tau_0\int_{\Omega_R} \varepsilon(\Grad\rho_{\al} \cdot \Grad)\vu_{\al} \cdot\vphi +\rho_\alpha\mathbf{f}\cdot\vphi \dydt
\end{align}
for all $\vphi\in X_n^R$ and any $\tau \in [0,T]$.

% As the anisotropic Lame operator $-\mu\Delta_x-\delta\partial_{zz}-(\mu+\lambda)\nabla\text{div}$ still enjoy the symmetric and strongly elliptic property, we can apply the \cite{n} Lemma 4.32 and 4.33 that there exists countable sets
% \begin{align}
% &\{\lambda_i\}^{\infty}_{i=1},\hspace{5pt}0<\lambda_1\leq\lambda_2\leq...,\\
% &\{\phi_i\}^{\infty}_{i=1}\in W^{1,p}\cap W^{2,p},\hspace{3pt}1\leq p<\infty,
% \end{align}
% such that

% \begin{align}
% -\mu\Delta_x\phi_i-\delta\partial_{zz}\phi_i-(\mu+\lambda)\nabla\rm{div}\phi_i=\lambda_i\phi_i
% \end{align}
% and $\{\phi_i\}$ is an orthonormal basis in $L^2(\Omega_R)$ and an orthogonal basis in $W^{1,2}(\Omega_R)$. Thus we denote $X_n=span\{\phi\}^n_{i=1}$ and set $\mathbf u\in C([0,T],X_n)$.

Now, the integral identity \eqref{eq:MEweak} can be rephrased for any $\tau\in [0,T]$ as
	\begin{equation*}
	\langle \mathcal{M}[\rho_{\al}(\tau, \cdot)] (\textbf{u}_{\al}(\tau, \cdot)), \vphi\rangle = \langle \textbf{J}_0^*, \vphi \rangle + \langle \int_{0}^{\tau} \mathcal{N} [\rho_{\al}(s, \cdot), \textbf{u}_{\al}(s,\cdot)] \ds, \vphi \rangle,
	\end{equation*}
with
	\begin{align*}
		\mathcal{M}[\rho]: X_n^R \rightarrow (X_n^R)^*, \quad &\langle \mathcal{M}[\rho] \textbf{u}, \vphi \rangle :=\int_{\Omega_R}\rho \textbf{u} \cdot \vphi \dy, \\
		\textbf{J}_0^* \in (X_n^R)^*, \quad &	\langle \textbf{J}_0^*, \vphi \rangle:= \int_{\Omega_R} \rho_{0,n,R}\vu_{0,n,R} \cdot \vphi  \dy, \\
		\mathcal{N}[\rho, \textbf{u}] \in (X_n^R)^*, \quad & \langle  \mathcal{N} [\rho, \textbf{u}], \vphi\rangle :=\int_{\Omega_R} \left[ (\rho \vu \otimes \vu) : \nabla_y\vphi +p(\rho)\Div \vphi -\mu\nabla_x\mathbf u:\nabla_x\vphi\right] \dy  \\
		& - \int_{\Omega_R}\left[\delta\partial_z\textbf{u}\cdot\partial_z\vphi + (\mu+\lambda)\Div\vu\ \Div\vphi + \varepsilon  (\nabla_y \rho \cdot \nabla_y) \vu \cdot \vphi+\rho\mathbf{f}\cdot\vphi\right] \dy.
	\end{align*}	
Then we define a mapping
\begin{align*}
\mathcal{T}:\mathbf u\in C([0,T],X_n^R)\rightarrow \mathcal{T}(\mathbf u)=\mathbf u(\rho)\in C([0,T],X_n^R),
	\end{align*}	
where
\begin{equation*}
		\mathcal{T}[\vu](\tau, \cdot):= \mathcal{M}^{-1} [\rho_{}(\tau, \cdot)] \left( \textbf{\textup{J}}_0^*+ \int_{0}^{\tau} \mathcal{N} [\rho_{}(s, \cdot), \vu(s,\cdot)] \ds \right),
		\end{equation*}
and we use the Leray-Schauder fixed point theorem to obtain $\vu_\al \in C([0,T^*];X_n^R)$ such that $\vu_\al = \mathcal{T}(\vu_\al)$ with $\rho = \rho_\al(\vu_\al)$ given by Proposition \ref{existence approximated densities}. It is then a standard matter to prolong the existence time from $T^*$ to $T$, see i.e. \cite[Proposition 7.34]{n}. Following \cite[Section 7.7.4.2]{n}, we have
\begin{equation}\label{4.17}
\begin{split}
&- \int_0^T\int_{\Omega_R}\left(\frac{1}{2}\rho_{\al}|\vu_{\al} |^2+P(\rho_{\al} )\right)\psi'(t) \dydt
+\int_0^T\int_{\Omega_R}\left(\mu|\nabla_x\vu_{\al} |^2+\delta|\partial_z\vu_{\al}|^2\right)\psi(t) \dydt \\
&\quad +\int_{\Omega_R}(\mu+\lambda)|\Div\vu_{\al} |^2\psi(t) \dydt +\ep\gamma\int_0^T\int_{\Omega_R}\rho_\al^{\gamma-2}|\Grad\rho_\al |^2\psi(t) \dydt \\
&\quad = \int_0^T\int_{\Omega_R}\rho_{\al}\mathbf{f}\cdot\vu_{\al}\psi(t) \dydt + \int_{\Omega_R}\left(\frac12 \rho_{0,n,R}|\vu_{0,n,R}|^2 + P(\rho_{0,n,R})\right)\psi(0) \dy,
\end{split}
\end{equation}
for all $\psi \in C^\infty_c([0,T))$. Since on a fixed bounded domain the total mass $\int_{\Omega_R} \rho_\alpha \dy$ is constant in time, which is an easy consequence of \eqref{approximation continuity equation} and \eqref{eq:approxBC1}-\eqref{eq:approxBC2}, we obtain
\begin{equation}\label{4.18}
- \int_0^T\int_{\Omega_R}\left(\rhoin^\gamma - \frac{\gamma}{\gamma-1}\rho_\al\rhoin^{\gamma-1}\right)\psi'(t) \dydt
=\int_{\Omega_R} \left(\rhoin^\gamma - \frac{\gamma}{\gamma-1}\rho_{0,n,R}\rhoin^{\gamma-1}\right)\psi(0) \dy.
\end{equation}
Adding \eqref{4.18} to \eqref{4.17} and using a sequence of test functions $\psi_k(t)$ approximating $1_{[0,\tau]}$ we obtain after passing to the limit
\begin{equation}\label{eq:hoho}
\begin{split}
&\int_{\Omega_R}\left(\frac{1}{2}\rho_{\al}|\vu_{\al} |^2+\mathcal{E}^{\rho_\infty}(\rho_{\al} )\right)(\tau,\cdot) \dy
+\int^\tau_0\int_{\Omega_R}\left(\mu|\nabla_x\vu_{\al} |^2+\delta|\partial_z\vu_{\al}|^2\right) \dydt \\
&\quad +\int^\tau_0\int_{\Omega_R}(\mu+\lambda)|\Div\vu_{\al} |^2 \dydt +\ep\gamma\int^\tau_0\int_{\Omega_R}\rho_\al^{\gamma-2}|\Grad\rho_\al |^2 \dydt \\
&\quad = \int_{\Omega_R}\left(\frac{1}{2}\rho_{0,n,R}|\vu_{0,n,R} |^2+\mathcal{E}^{\rhoin}(\rho_{0,n,R} )\right) \dy
+\int^\tau_0\int_{\Omega_R}\rho_{\al}\mathbf{f}\cdot\vu_{\al} \dydt
\end{split}
\end{equation}
for all $\tau \in (0,T)$.

Next, we introduce the essential and residual parts of any function $h$ related to $\rho_\al$. We denote
\begin{equation}
    [h]_{ess,\rho_\al} = h 1_{\rho_\al \in (\rhoin / 2, 2\rhoin)}, \qquad [h]_{res,\rho_\al} = h - [h]_{ess,\rho_\al}.
\end{equation}
In order to shorten notation, we will use $[h]_{ess}$ instead of $[h]_{ess,\rho_\al}$ and similarly for the residual part if no confusion arises. It is easy to verify that
\begin{equation}\label{eq:Pass}
\lim_{s\rightarrow\rho_\infty}\frac{2}{\gamma\rho^{\gamma-2}_\infty}\frac{\mathcal{E}^{\rho_\infty}(s)}{(s-\rho_\infty)^2}=1,
\hspace{8pt}\lim_{s\rightarrow\infty}(\gamma-1)\frac{\mathcal{E}^{\rho_\infty}(s)}{(s-\rho_\infty)^\gamma}=1.
\end{equation}
Then we deduce
\begin{align}\label{4.19}
&c_1\mathcal{E}^{\rho_\infty}(\rho) \leq [1]_{ess}|\rho-\rho_\infty|^2 \leq C_1\mathcal{E}^{\rho_\infty}(\rho), \\ \label{4.20}
&c_2\mathcal{E}^{\rho_\infty}(\rho)\leq [1]_{res}|\rho-\rho_\infty|^\gamma
\leq C_2\mathcal{E}^{\rho_\infty}(\rho).
\end{align}
Next, we write
\begin{equation}\label{eq:RHSsplit}
\begin{split}
&\int^\tau_0\int_{\Omega_R}\rho_{\al}\mathbf{f}\cdot\vu_{\al} \dydt = \int^\tau_0\int_{\Omega_R}([1]_{ess}(\rho_\al-\rho_\infty)+\rho_\infty)\mathbf f\cdot\mathbf u_{\al}\dydt \\
& \quad
+\int^\tau_0\int_{\Omega_R}[1]_{res}(\rho_\al-\rho_\infty)\mathbf f\cdot\mathbf u_{\al}\dydt
\end{split}
\end{equation}
We estimate the last term on the right hand side of \eqref{eq:hoho} using \eqref{eq:RHSsplit} as follows.
\begin{equation}\label{eq:RHSsplit1}
\begin{split}
&\left|\int^\tau_0\int_{\Omega_R}([1]_{ess}(\rho_\al-\rho_\infty)+\rho_\infty)\mathbf f\cdot\mathbf u_{\al}\dydt\right|\\
&\quad \leq2\rho_\infty\|\mathbf f\|_{L^2(0,T;L^{\frac{6}{5}}(\Omega_R))}\|\mathbf u_\al\|_{L^2(0,\tau;L^6(\Omega_R))}\\
&\quad \leq C(\beta)\|\mathbf f\|^2_{L^2(0,T;L^{\frac{6}{5}}(\Omega_R))}
+\frac{\beta}{4}\|\nabla_y\mathbf u_\al\|_{L^2((0,\tau)\times\Omega_R)}^2
\end{split}
\end{equation}
and
\begin{equation}\label{eq:RHSsplit2}
\begin{split}
&\left|\int^\tau_0\int_{\Omega_R}[1]_{res}(\rho_\al-\rho_\infty)\mathbf f\cdot\mathbf u_{\al}\dydt \right| \\
&\quad \leq\int^\tau_0\int_{\Omega_R}[1]_{res}(\sqrt{\rho_\al}+\sqrt{\rho_\infty})\sqrt{|\rho_\al-\rho_\infty|}|\mathbf f||\mathbf u_{\al}|\dydt\\
&\quad \leq\int^\tau_0\|\sqrt{\rho_\al}\mathbf u_\al\|_{L^{2}(\Omega_R)}
\|[1]_{res}\sqrt{|\rho_\al-\rho_\infty|}\|_{L^{2\gamma}(\Omega_R)}\|\mathbf f\|_{L^{\frac{2\gamma}{\gamma-1}}(\Omega_R)} \dt\\
&\quad +\sqrt{\rho_\infty}\int_0^\tau\|[1]_{res}\sqrt{|\rho_\al-\rho_\infty|}\|_{L^{2\gamma}(\Omega_R)}
\|\mathbf u_\al\|_{L^{6}(\Omega_R)}\|\mathbf f\|_{L^{\frac{6\gamma}{5\gamma-3}}(\Omega_R)} \dt \\
&\quad \leq \int^\tau_0\int_{\Omega_R}\left(\frac{1}{2}\rho_\al|\mathbf u_\al|^2 + \mathcal{E}^{\rho_\infty}(\rho_\al)\right) \dydt +\frac{\beta}{4}\|\nabla_y\mathbf u_\al\|_{L^2((0,\tau)\times\Omega_R)}^2 \\
&\quad + C\|\vf\|_{L^{\frac{2\gamma}{\gamma-1}}((0,T)\times \Omega_R)}^{\frac{2\gamma}{\gamma-1}}
 + C(\beta)\|\vf\|_{L^{\frac{2\gamma}{\gamma-1}}(0,T;L^{\frac{6\gamma}{5\gamma-3}}(\Omega_R))}^{\frac{2\gamma}{\gamma-1}},
\end{split}
\end{equation}
where we recall that the constant $\beta$ is defined in \eqref{eq:C0def} and satisfies Corollary \ref{co:beta}. We conclude the uniform estimates by means of the Gronwall lemma and recover
\begin{equation}\label{eq:hoho_final}
\begin{split}
&\int_{\Omega_R}\left(\frac{1}{2}\rho_{\al}|\vu_{\al} |^2+\mathcal{E}^{\rho_\infty}(\rho_{\al} )\right)(\tau,\cdot) \dy
+\int^\tau_0\int_{\Omega_R}\left(\mu|\nabla_x\vu_{\al} |^2+\delta|\partial_z\vu_{\al}|^2\right) \dydt \\
&\quad +\int^\tau_0\int_{\Omega_R}(\mu+\lambda)|\Div\vu_{\al} |^2 \dydt +\ep\gamma\int^\tau_0\int_{\Omega_R}\rho_\al^{\gamma-2}|\Grad\rho_\al |^2 \dydt \leq C(\rho_{0,n,R},\vu_{0,n,R},\vf)
\end{split}
\end{equation}
for all $\tau \in (0,T)$.

We summarize the obtained properties in the following Proposition.
\begin{proposition}\label{prop:sol_ep_n_R}
Let $n,R \in \N$ and $\varepsilon > 0$. Let $\vr_{0,n} \in C^1(\overline{\Omega})$ satisfy \eqref{eq:approxICprop} and $\mathbf u_{0,n,R}\in X_n^R$. Then, there exists a solution $\rho_\al$ and $\vu_{\al}$ to \eqref{4.1}--\eqref{eq:approxIC2} on time interval $(0,T)$ such that
\begin{align*}
&\rho_{\al}\in L^2(0,T;W^{1,2}(\Omega_R))\cap L^2(0,T;W^{2,2}(\Omega_R)),\\
&\partial_t\rho_{\al}\in L^2(0,T;L^2(\Omega_R)), \quad \varepsilon|\nabla\rho_{\al}|^2\in L^1(0,T;L^1(\Omega_R)), \\
&\vu_\al\in L^\infty(0,T;X_n^R), \quad \partial_t\vu_\al\in L^2(0,T;X_n^R).
\end{align*}
Moreover, it holds
\begin{equation*}
0 < \underline{\rho_n}\exp \left( -T \|\Div \vu_{\al}\|_{L^{\infty}((0,T) \times \Omega_R)} \right)  \leq \rho_\al(t,y) \leq \overline{\rho_n}\exp \left( T \|\Div \vu_{\al}\|_{L^{\infty}((0,T) \times \Omega_R)} \right),
\end{equation*}
for all $(t,y) \in [0,T]\times \overline{\Omega_R}$ and
\begin{equation}\label{eq:EIvar1}
\begin{split}
&\int_{\Omega_R}\left(\frac{1}{2}\rho_{\al}|\vu_{\al} |^2+\mathcal{E}^{\rho_\infty}(\rho_{\al} )\right)(\tau,\cdot) \dy
+\int^\tau_0\int_{\Omega_R}\left(\mu|\nabla_x\vu_{\al} |^2+\delta|\partial_z\vu_{\al}|^2\right) \dydt \\
&\quad +\int^\tau_0\int_{\Omega_R}(\mu+\lambda)|\Div\vu_{\al} |^2 \dydt +\ep\gamma\int^\tau_0\int_{\Omega_R}\rho_\al^{\gamma-2}|\Grad\rho_\al |^2 \dydt \\
&\quad \leq \int_{\Omega_R}\left(\frac{1}{2}\rho_{0,n,R}|\vu_{0,n,R} |^2+\mathcal{E}^{\rhoin}(\rho_{0,n,R} )\right) \dy
+\int^\tau_0\int_{\Omega_R}\rho_{\al}\mathbf{f}\cdot\vu_{\al} \dydt
\end{split}
\end{equation}
for all $\tau \in (0,T)$.
\end{proposition}

%%%%%%%%%%%%%%%%%%%%%%%%%%%%%%%%%%%%%%%%%%%%%%%%%%%%%%%%%%%%%%%%%%%%%%%%%%%%%%%%%%%%%%%%%%%%%%%%%%%%%%%%%%%%%%%%%%%%
\subsection{The limit $\varepsilon\rightarrow0$}\label{ss:43}
In this section, we keep fixed $n,R \in \N$ and let $\varepsilon\rightarrow0$. In order to emphasize the dependence of the sequence of solutions on $\ep$, we denote the functions constructed in Proposition \ref{prop:sol_ep_n_R} as $\rho_{\ep}$, $\vu_\ep$.

Since $n$ remains fixed and all norms are equivalent in the finitely dimensional spaces, we have
\begin{equation*}
\|\vu_{\varepsilon} \|_{L^\infty(0,T;W^{1,\infty}(\Omega_R))}\leq C,
\end{equation*}
and consequently
\begin{equation*}
0<\underline{\rho}\leq\rho_{\varepsilon} (t,y)\leq\overline{\rho}
\end{equation*}
for some $0 < \underline{\rho} < \overline{\rho}$ independent of $\varepsilon$. This implies using \eqref{eq:EIvar1} that
\begin{align*}
\varepsilon\|\nabla\rho_{\varepsilon} \|^2_{L^2(0,T;L^{2}(\Omega_R))}\leq C.
\end{align*}
From the continuity equation \eqref{approximation continuity equation} we deduce
\begin{equation*}
\begin{split}
&\rho_{\varepsilon} \rightarrow\rho \quad \text{weakly}^* \text{ in } L^\infty((0,T)\times\Omega_R), \\
&\rho_{\varepsilon} \rightarrow\rho \quad \text{ in }  C_{weak}([0,T];L^r(\Omega_R)) \text{ for any } 1<r<\infty, \end{split}
\end{equation*}
% {\color{red}
% \begin{align*}
% &\|1_{|\rho-\rho_\infty|\geq1}(\rho_\al-\rho_\infty)\|_{L^2(0,T);L^\gamma(\Omega_R))}\leq C,
% \hspace{5pt}\|1_{|\rho-\rho_\infty|<1}(\rho_\al-\rho_\infty)\|_{L^2(0,T);L^2(\Omega_R))}\leq C,\\
% &\rho_{\varepsilon} \rightarrow\rho \quad \text{weakly}^* \text{ in } L^\infty((0,T)\times\Omega_R)\quad  \text{and weakly in }  C_{weak}([0,T];L^r(\Omega_R)) \text{ for any } 1<r<\infty,
% \end{align*}  }
and the limit density still satisfies
\begin{equation*}
0<\underline{\rho}\leq\rho (t,y)\leq\overline{\rho}.
\end{equation*}
Similarly,
\begin{align*}
&\vu_{\varepsilon} \rightarrow\vu \quad \text{weakly}^* \text{ in } L^\infty(0,T;W^{1,\infty}(\Omega_R)),\\
&\rho_{\varepsilon} \vu_{\varepsilon} \rightarrow\mathbf m \quad \text{weakly}^*\text{ in } L^\infty((0,T)\times\Omega_R), \\
&\rho_{\varepsilon} \vu_{\varepsilon} \rightarrow\mathbf m \quad \text{ in } C_{weak}([0,T];L^r(\Omega_R)) \text{ for any } 1<r<\infty.
\end{align*}
Finally, we use the Arzel\'a-Ascoli theorem to identify
\begin{align*}
\mathbf m=\rho\vu \quad \text{ a.a in } (0,T)\times\Omega_R.
\end{align*}
Now we are ready to pass to the limit in the equations of motion. It is straightforward to pass to the limit in the continuity equation \eqref{eq:CEweak} to obtain
\begin{equation}\label{eq:CEweak_lim1}
\int_{\Omega_R}(\rho\varphi)(\tau,y) \dy - \int_{\Omega_R}\rho_{0,n,R}(y)\varphi(0,y) \dy = \int^\tau_0\int_{\Omega_R}\left(\rho\partial_t\varphi + \rho\vu\cdot\Grad\varphi\right) \dydt,
\end{equation}
for all $\varphi\in C^1([0,T]\times\overline{\Omega_R})$ and almost all $\tau \in [0,T]$.

Next, we pass to the limit in the momentum equation \eqref{eq:MEweak}. The energy inequality \eqref{eq:EIvar1} together with Lemma \ref{l:elip} implies
\begin{equation*}
\|\nabla_y\vu_{\varepsilon} \|_{L^2((0,T)\times \Omega_R)}\leq C,
\end{equation*}
and therefore
\begin{align*}
&\nabla_x\vu_{\varepsilon} \rightarrow\nabla_x\vu \quad  \text{ weakly in } L^2((0,T)\times \Omega_R),\\
&\partial_z\vu_{\varepsilon}\rightarrow \partial_z\vu \quad \text{ weakly in } L^2((0,T)\times \Omega_R), \\
&\Div\vu_{\varepsilon}\rightarrow \Div\vu \quad \text{ weakly in } L^2((0,T)\times \Omega_R).
\end{align*}
Denoting by $\Pi_n: L^2(\Omega_R) \to X_n^R$ the orthogonal projection, we deduce from the momentum equation \eqref{eq:MEweak}
that $\partial_t\Pi_n[\rho_{\varepsilon} \vu_{\varepsilon}]$ is bounded in $L^2(0,T;X_n^R)$. This allows us to conclude that
\begin{equation*}
    \rho_\ep \vu_\ep \otimes \vu_\ep \to \rho \vu \otimes \vu \quad \text{ weakly}^* \text{ in } L^\infty((0,T)\times \Omega_R).
\end{equation*}
%
% From the continuity equation \eqref{eq:CEweak}, we can obtain $\partial_t\rho_{\varepsilon,n}$ bounded in $L^2((0,T)\times\Omega_R)$ similarly. Consequently, we get $\|\partial_t\mathbf u_{\varepsilon,n}\|_{L^2((0,T)\times \Omega_R)}\leq C$.
%
% Due to Arzela-Ascoli compactness argument, we assume
% \begin{align*}
% &\mathbf u_{\varepsilon,n}\rightarrow\mathbf u_n\hspace{3pt} \text{in} \hspace{3pt} C((0,T)\times \Omega_R),\\
% &\rho_{\varepsilon,n}\mathbf u_{\varepsilon,n}\rightarrow\rho_n\mathbf u_n, \hspace{3pt}
% \rho_{\varepsilon,n}\mathbf u_{\varepsilon,n}\otimes\mathbf u_{\varepsilon,n}\rightarrow\rho_n\mathbf u_n\otimes\mathbf u_n
% \hspace{3pt}\text{in}\hspace{3pt} C_{weak}((0,T);L^q(\Omega_R)), 1\leq q<\infty.
% \end{align*}
By virtue of Proposition 4.2, we get $\|p(\rho_{\varepsilon})\|_{L^\infty((0,T)\times\Omega_R)}$ is bounded uniformly in $\varepsilon$. Thus,
\begin{equation}\label{eq:press1}
p(\rho_{\varepsilon})\rightarrow\overline{p(\rho)} \quad  \text{weakly}^* \text{ in } L^\infty((0,T)\times \Omega_R).
\end{equation}
Since $p(\rho)$ is continuous and convex, and $\rho_{\varepsilon}\rightarrow \rho$ weakly in $L^1((0,T)\times\Omega_R)$, we get
\begin{equation}\label{eq:press2}
p(\rho)\leq\overline{p(\rho)} \quad \text{ a.e in } (0,T)\times\Omega_R.
\end{equation}
We define the Reynolds defect measure at this approximation level as
\begin{equation*}
\mathfrak{R}(t) = (\overline{p(\rho)} - p(\rho))\mathbb{I}.
\end{equation*}
Passing to the limit in \eqref{eq:MEweak} while moving to test functions depending also on time we end up with
\begin{equation}\label{eq:MEweak_lim1}
\begin{split}
&\int_{\Omega_R}(\rho \vu \cdot\vphi)(\tau,y)\dy - \int_{\Omega_R}(\rho_{0,n,R} \vu_{0,n,R})(y) \cdot\vphi(0,y) \dy \\
& \quad = \int^\tau_0\int_{\Omega_R}\left(\rho \vu\cdot \partial_t\vphi + \rho \vu \otimes\vu :\Grad\vphi
+p(\rho)\Div\vphi\right)\dydt + \int^\tau_0\int_{\Omega_R} \Grad \vphi : \d\mathfrak{R}(t)\dt   \\
&-\int^\tau_0\int_{\Omega_R}\left(\mu\nabla_x \vu :\nabla_x\vphi + \delta\partial_z\vu \cdot \partial_z\vphi
+(\mu+\lambda)\Div\vu\, \Div\vphi\right) \dydt
\end{split}
\end{equation}
for all $\vphi\in C^1([0,T];X_n^R)$ and almost all $\tau \in [0,T]$.

We continue with the limit process in the energy inequality \eqref{eq:EIvar1}. First, the same way as in \eqref{eq:press1}-\eqref{eq:press2} we get also
\begin{equation}\label{eq:Press1}
\mathcal{E}^{\rho_\infty}(\rho_{\varepsilon})\rightarrow\overline{\mathcal{E}^{\rho_\infty}(\rho)} \quad  \text{weakly}^* \text{ in } L^\infty((0,T)\times \Omega_R),
\end{equation}
and
\begin{equation}\label{eq:Press2}
\mathcal{E}^{\rho_\infty}(\rho)\leq\overline{\mathcal{E}^{\rho_\infty}(\rho)} \quad \text{ a.e in } (0,T)\times\Omega_R.
\end{equation}
Defining the energy defect measure at this approximation level as
\begin{equation*}
\mathfrak{E}(t) = (\overline{\mathcal{E}^{\rho_\infty}(\rho)} - \mathcal{E}^{\rho_\infty}(\rho))\mathbb{I},
\end{equation*}
using weak lower semicontinuity of convex functions (see Lemma \ref{l:convex}) and discarding the nonnegative term containing $\Grad\rho_\ep$, we pass to the limit in \eqref{eq:EIvar1} and obtain
\begin{equation}\label{eq:EIvar2}
\begin{split}
&\int_{\Omega_R}\left(\frac{1}{2}\rho|\vu |^2+ \mathcal{E}^{\rho_\infty}(\rho )\right)(\tau,\cdot) \dy + \int_{\Omega_R} \d\mathfrak{E}(\tau) \\
&\quad +\int^\tau_0\int_{\Omega_R}\left(\mu|\nabla_x\vu |^2+\delta|\partial_z\vu|^2+ (\mu+\lambda)|\Div\vu |^2\right) \dydt  \\
&\quad\leq \int_{\Omega_R}\left(\frac{1}{2}\rho_{0,n,R}|\vu_{0,n,R} |^2+\mathcal{E}^{\rho_\infty}(\rho_{0,n,R} )\right) \dy + \int_0^\tau\int_{\Omega_R} \rho\vf\cdot\vu \dydt
\end{split}
\end{equation}
for almost all $\tau \in (0,T)$. Note that \eqref{eq:pressurelaw} together with \eqref{eq:PP} implies
\begin{equation}\label{eq:RErel}
    \tr [\mathfrak{R}] = 3(\gamma-1) \mathfrak{E}.
\end{equation}

In what follows, we will pass with $n \to \infty$. Therefore we denote the limits obtained in this section as $\rho_n, \vu_n, \mathfrak{R}_n$ and $\mathfrak{E}_n$ and we rewrite the limit equations \eqref{eq:CEweak_lim1}, \eqref{eq:MEweak_lim1} and \eqref{eq:EIvar2} as
\begin{equation}\label{eq:CEweak_lim11}
\int_{\Omega_R}(\rho_n\varphi)(\tau,y) \dy - \int_{\Omega_R}\rho_{0,n,R}(y)\varphi(0,y) \dy = \int^\tau_0\int_{\Omega_R}\left(\rho_n\partial_t\varphi + \rho_n\vu_n\cdot\Grad\varphi\right) \dydt,
\end{equation}
for all $\varphi\in C^1([0,T]\times\overline{\Omega_R})$ and almost all $\tau \in [0,T]$.
\begin{equation}\label{eq:MEweak_lim11}
\begin{split}
&\int_{\Omega_R}(\rho_n \vu_n \cdot\vphi)(\tau,y)\dy - \int_{\Omega_R}(\rho_{0,n,R} \vu_{0,n,R})(y) \cdot\vphi(0,y) \dy  = \int^\tau_0\int_{\Omega_R} \Grad \vphi : \d\mathfrak{R}_n(t)\dt\\
& \quad + \int^\tau_0\int_{\Omega_R}\left(\rho_n \vu_n\cdot \partial_t\vphi + \rho_n \vu_n \otimes\vu_n :\Grad\vphi
+p(\rho_n)\Div\vphi\right)\dydt    \\
&\quad -\int^\tau_0\int_{\Omega_R}\left(\mu\nabla_x \vu_n :\nabla_x\vphi + \delta\partial_z\vu_n \cdot \partial_z\vphi
+(\mu+\lambda)\Div\vu_n\, \Div\vphi+\rho_n\vf\cdot\vphi\right) \dydt
\end{split}
\end{equation}
for all $\vphi\in C^1([0,T];X_n^R)$ and almost all $\tau \in [0,T]$.
\begin{equation}\label{eq:EIvar21}
\begin{split}
&\int_{\Omega_R}\left(\frac{1}{2}\rho_n|\vu_n |^2+ \mathcal{E}^{\rho_\infty}(\rho_n )\right)(\tau,\cdot) \dy + \int_{\Omega_R} \d\mathfrak{E}_n(\tau) \\
&\quad +\int^\tau_0\int_{\Omega_R}\left(\mu|\nabla_x\vu_n |^2+\delta|\partial_z\vu_n|^2 +(\mu+\lambda)|\Div\vu_n |^2\right) \dydt
\\
&\quad \leq \int_{\Omega_R}\left(\frac{1}{2}\rho_{0,n,R}|\vu_{0,n,R} |^2+\mathcal{E}^{\rho_\infty}(\rho_{0,n,R} )\right) \dy + \int_0^\tau\int_{\Omega_R} \rho_n\vf\cdot\vu_n \dydt
\end{split}
\end{equation}
for almost all $\tau \in (0,T)$.

\subsection{The limit $n\rightarrow\infty$}\label{s:limit_n}

Here, we pass with $n \to \infty$ while keeping the problem restricted to a bounded domain, i.e. keeping $R$ fixed. We start with deducing the a priori estimates from \eqref{eq:EIvar21}. We repeat the arguments leading to \eqref{eq:hoho_final} using the essential and residual parts $[h]_{ess,\rho_n}$ and $[h]_{res,\rho_n}$ in the estimate of the last term on the right hand side of \eqref{eq:EIvar21}. Moreover, we recall that we assumed $\rho_{0,n,R}|\vu_{0,n,R}|^2 \to \frac{|(\rho\vu)_{0,R}|^2}{\vr_{0,R}}$ in $L^1(\Omega_R)$ and $\rho_{0,n,R} \to \rho_{0,R}$ in $L^\gamma(\Omega)$, so the terms containing initial data on the right hand side of \eqref{eq:EIvar21} are also bounded. As the domain $\Omega_R$ remains bounded, we read the bounds
\begin{align*}
    &\|\rho_n\|_{L^\infty(0,T;L^\gamma(\Omega_R))} \leq C, \\
    &\|\sqrt{\rho_n} \vu_n\|_{L^\infty(0,T;L^2(\Omega_R))} \leq C, \\
    &\|\vu_n\|_{L^2(0,T;W^{1,2}_0(\Omega_R))} \leq C, \\
    &\|\mathfrak{E}_n\|_{L^\infty(0,T;\mathcal{M}(\overline{\Omega_R}))} \leq C.
\end{align*}
Therefore, passing to subsequences, we obtain
\begin{align*}
    &\rho_n \to \rho \quad \text{ in } C_{weak}([0,T];L^\gamma(\Omega_R)), \\
    &\vu_n \to \vu \quad \text{ weakly in } L^2(0,T;W^{1,2}_0(\Omega_R)) ,\\
    & \rho_n \vu_n \to \mathbf{m} \quad \text{ in } C_{weak}([0,T];L^{\frac{2\gamma}{\gamma+1}}(\Omega_R)), \\
    &\mathfrak{E}_n \to \mathfrak{E}_1 \quad \text{ weakly}^* \text{ in } L^\infty(0,T;\mathcal{M}(\overline{\Omega_R})), \\
    &\mathfrak{R}_n \to \mathfrak{R}_1 \quad \text{ weakly}^* \text{ in } L^\infty(0,T;\mathcal{M}_{3\times 3}(\overline{\Omega_R})).
\end{align*}

In order to identify $\mathbf{m} = \rho \vu$, we need the following result, which is based on the Div-Curl Lemma, see \cite[Lemma 8.1]{e5}).
\begin{lemma}\label{l:rhou_iden}
    Let $Q_R = (0,T)\times\Omega_R$. Suppose that for some $p,q,r > 1$ it holds
    \begin{align*}
        &r_n \to r \quad \text{ weakly in } L^p(Q_R), \\
        &v_n \to v \quad \text{ weakly in } L^q(Q_R), \\
        &r_n v_n \to w \quad \text{ weakly in } L^r(Q_R).
    \end{align*}
    Moreover let
    \begin{equation*}
        \de_t r_n = \Div \mathbf{g}_n %+ h_n
        \quad \text{ in } \mathcal{D}'(Q_R), \quad \|\mathbf{g}_n\|_{L^s(Q_R)} \leq C
    \end{equation*}
    for some $s > 1$.
    % and
    % \begin{equation*}
    %     h_n \text{ precompact in } W^{-1,z}
    % \end{equation*}
    % for some $z > 1$.
    Finally let
    \begin{equation*}
        \| \Grad v_n \|_{\mathcal{M}(\overline{Q_R})} \leq C \quad \text{ uniformly for } n \to \infty.
    \end{equation*}
    Then $w = rv$ a.e. in $Q_R$.
\end{lemma}
Applying this lemma to $r_n = \rho_n$, $v_n = [\vu_n]_i$, $\mathbf{g}_n = -\rho_n \vu_n$, $p = \gamma$, $q = 2$, $r = s = \frac{2\gamma}{\gamma+1}$ yields the desired claim $\mathbf{m} = \rho\vu$ a.e. in $(0,T)\times \Omega_R$.

This is all we need to pass to the limit in the continuity equation \eqref{eq:CEweak_lim11} and recover
\begin{equation}\label{eq:CEweak_lim2}
\int_{\Omega_R}(\rho\varphi)(\tau,y) \dy - \int_{\Omega_R}\rho_{0,R}(y)\varphi(0,y) \dy = \int^\tau_0\int_{\Omega_R}\left(\rho\partial_t\varphi + \rho\vu\cdot\Grad\varphi\right) \dydt,
\end{equation}
for all $\varphi\in C^1([0,T]\times\overline{\Omega_R})$ and almost all $\tau \in [0,T]$.

For any $0\neq\xi\in\mathbb{R}^3$ we introduce the following convex lower semicontinuous functions
	\begin{equation*}
		E_\xi[r,\mathbf m] =\begin{cases}
		\frac{|\mathbf m\cdot\xi|^2}{r} &\mbox{if}\hspace{3pt} r>0, \\
		0 &\mbox{if } \mathbf m=0,\\
        \infty & \mbox{otherwise},
		\end{cases}
	\end{equation*}
together with function
\begin{equation*}
		E[r,\mathbf m] =\begin{cases}
		\frac{\mathbf m\otimes\mathbf m}{r} &\mbox{if}\hspace{3pt} r>0, \\
		0 &\mbox{if } \mathbf m=0,\\
        \infty & \mbox{otherwise}.
		\end{cases}
		\end{equation*}
 and
 \begin{equation*}
 	E_0[r,\vm]=E_{e_1}[r,\vm]+E_{e_2}[r,\vm]+E_{e_3}[r,\vm]=\begin{cases}
		\frac{|\mathbf m|^2}{r} &\mbox{if}\hspace{3pt} r>0, \\
		0 &\mbox{if } \mathbf m=0,\\
        \infty & \mbox{otherwise}.
		\end{cases}
\end{equation*}

Now we can write
\begin{equation*}
    \rho_n\vu_n \otimes \vu_n = \frac{\vm_n \otimes \vm_n}{\rho_n} = E[\rho_n, \vm_n]\to \overline{E[\vr,\vm]} \quad \text{ weakly}^* \text{ in } L^\infty(0,T;\mathcal{M}_{3\times 3}(\overline{\Omega_R})),
\end{equation*}
and
\begin{equation*}
    p(\rho_n) \to \overline{p(\rho)} \quad \text{ weakly}^* \text{ in } L^\infty(0,T;\mathcal{M}(\overline{\Omega_R})).
\end{equation*}
We observe that
\begin{equation*}
    \mathfrak{R}_2 = \overline{E[\vr,\vm]} - E[\vr,\vm] \geq 0,
\end{equation*}
meaning that the left hand side is a positive semidefinite matrix. Indeed, observe that
\begin{equation*}
    \mathfrak{R}_2 : (\xi\otimes\xi) = \overline{E_\xi[\rho,\vm]} - E_\xi[\rho,\vm] \geq 0
\end{equation*}
as $E_\xi$ is a convex lower semicontinuous function. Denoting also
\begin{equation*}
    \mathfrak{R}_3 = (\overline{p(\rho)} - p(\rho))\mathbb{I} \geq 0
\end{equation*}
and
\begin{equation*}
    \mathfrak{R} = \mathfrak{R}_1 + \mathfrak{R}_2 + \mathfrak{R}_3 \geq 0,
\end{equation*}
we may pass to the limit in the momentum equation \eqref{eq:MEweak_lim11} and recover
\begin{equation}\label{eq:MEweak_lim2}
\begin{split}
&\int_{\Omega_R}(\rho \vu \cdot\vphi)(\tau,y)\dy - \int_{\Omega_R}(\rho \vu)_{0,R}(y) \cdot\vphi(0,y) \dy  = \int^\tau_0\int_{\Omega_R} \Grad \vphi : \d\mathfrak{R}(t)\dt\\
& \quad + \int^\tau_0\int_{\Omega_R}\left(\rho \vu\cdot \partial_t\vphi + \rho \vu \otimes\vu :\Grad\vphi
+p(\rho)\Div\vphi\right)\dydt    \\
&\quad -\int^\tau_0\int_{\Omega_R}\left(\mu\nabla_x \vu :\nabla_x\vphi + \delta\partial_z\vu \cdot \partial_z\vphi
+(\mu+\lambda)\Div\vu\, \Div\vphi+\rho \vf\cdot\vphi\right) \dydt
\end{split}
\end{equation}
for all $\vphi\in C^1([0,T];X_n^R)$ with arbitrary $n \in \mathbb{N}$ and almost all $\tau \in [0,T]$. Using a density argument we moreover conclude that \eqref{eq:MEweak_lim2} holds for all $\vphi\in C^1([0,T];C^1_c(\Omega_R))$.

Finally, we pass to the limit in the energy inequality \eqref{eq:EIvar21}. Using similar arguments as in passing to the limit in the momentum equation we get
\begin{equation*}
    \rho_n|\vu_n|^2 = \frac{|\vm_n|^2}{\rho_n} = E_0[\rho_n, \vm_n]\to \overline{E_0[\vr,\vm]} \quad \text{ weakly}^* \text{ in } L^\infty(0,T;\mathcal{M}(\overline{\Omega_R})),
\end{equation*}
and
\begin{equation*}
    \mathcal{E}^{\rho_\infty}(\rho_n) \to \overline{\mathcal{E}^{\rho_\infty}(\rho)} \quad \text{ weakly}^* \text{ in } L^\infty(0,T;\mathcal{M}(\overline{\Omega_R}))
\end{equation*}
with
\begin{equation*}
    \mathfrak{E}_2 = \overline{E_0[\vr,\vm]} - E_0[\vr,\vm] \geq 0
\end{equation*}
and
\begin{equation*}
    \mathfrak{E}_3 = \overline{\mathcal{E}^{\rho_\infty}(\rho)} - \mathcal{E}^{\rho_\infty}(\rho) \geq 0,
\end{equation*}
whence
\begin{equation*}
    \mathfrak{E} = \mathfrak{E}_1 + \mathfrak{E}_2 + \mathfrak{E}_3 \geq 0.
\end{equation*}
Since
\begin{equation*}
    \tr[\mathfrak{R}_1 + \mathfrak{R}_3] = 3(\gamma-1)(\mathfrak{E}_1 + \mathfrak{E}_3)
\end{equation*}
and
\begin{equation*}
    \tr[\mathfrak{R}_2] = \mathfrak{E}_2,
\end{equation*}
it holds
\begin{equation}\label{eq:REineq1}
    \min\left\{1,3(\gamma-1)\right\}\mathfrak{E} \leq \tr[\mathfrak{R}] \leq \max\left\{1,3(\gamma-1)\right\}\mathfrak{E}.
\end{equation}
We use the convexity of the dissipative terms (see Lemma \ref{l:convex}) to conclude
\begin{equation}\label{eq:EIvar3}
\begin{split}
&\int_{\Omega_R}\left(\frac{1}{2}\rho|\vu |^2+ \mathcal{E}^{\rho_\infty}(\rho )\right)(\tau,\cdot) \dy + \int_{\Omega_R} \d\mathfrak{E}(\tau) \\
&\quad +\int^\tau_0\int_{\Omega_R}\left(\mu|\nabla_x\vu |^2+\delta|\partial_z\vu|^2 + (\mu+\lambda)|\Div\vu |^2\right) \dydt \\
&\quad \leq \int_{\Omega_R}\left(\frac{1}{2}\frac{|(\rho\vu)_{0,R}|^2}{\rho_{0,R}} +\mathcal{E}^{\rho_\infty}(\rho_{0,R} )\right) \dy + \int_0^\tau\int_{\Omega_R} \rho\vf\cdot\vu \dydt
\end{split}
\end{equation}
for almost all $\tau \in (0,T)$.

In the next section, we will pass with $R \to \infty$. Therefore we denote the limits obtained in this section as $\rho_R, \vu_R, \mathfrak{R}_R$ and $\mathfrak{E}_R$. We extend $\rho_R$ by $\rho_\infty$ and $\vu_R, \mathfrak{R}_R$ and $\mathfrak{E}_R$ by zero from their domain of definition $\Omega_R$ to $\Omega$. This way, we rewrite the limit equation \eqref{eq:CEweak_lim2} as
\begin{equation}\label{eq:CEweak_lim21}
\int_{\Omega}(\rho_R\varphi)(\tau,y) \dy - \int_{\Omega}\rho_{0,R}(y)\varphi(0,y) \dy = \int^\tau_0\int_{\Omega}\left(\rho_R\partial_t\varphi + \rho_R\vu_R\cdot\Grad\varphi\right) \dydt,
\end{equation}
for all $\varphi\in C^1_c([0,T]\times\overline{\Omega})$ and almost all $\tau \in [0,T]$. For a fixed test function $\varphi$ this formulation however holds only for $R$ satisfying the condition $\supp\ \varphi(t,\cdot) \subset \Omega_R$ for all $t \in [0,T]$.

Similarly, we rewrite \eqref{eq:MEweak_lim2} as
\begin{equation}\label{eq:MEweak_lim21}
\begin{split}
&\int_{\Omega}(\rho_R \vu_R \cdot\vphi)(\tau,y)\dy - \int_{\Omega}(\rho\vu)_{0,R}(y) \cdot\vphi(0,y) \dy  = \int^\tau_0\int_{\Omega} \Grad \vphi : \d\mathfrak{R}_R(t)\dt\\
& \quad + \int^\tau_0\int_{\Omega}\left(\rho_R \vu_R\cdot \partial_t\vphi + \rho_R \vu_R \otimes\vu_R :\Grad\vphi
+p(\rho_R)\Div\vphi\right)\dydt    \\
&\quad -\int^\tau_0\int_{\Omega}\left(\mu\nabla_x \vu_R :\nabla_x\vphi + \delta\partial_z\vu_R \cdot \partial_z\vphi
+(\mu+\lambda)\Div\vu_R\, \Div\vphi+\rho_R\vf\cdot\vphi\right) \dydt
\end{split}
\end{equation}
for all $\vphi\in C^1([0,T];C^1_c(\Omega))$ and almost all $\tau \in [0,T]$. Again, for a fixed test function $\vphi$ this formulation holds only for sufficiently large $R$, namely for $R$ satisfying $\supp\ \vphi(t,\cdot) \subset \Omega_R$ for all $t \in [0,T]$.

Finally, we rewrite \eqref{eq:EIvar3} as
\begin{equation}\label{eq:EIvar31}
\begin{split}
&\int_{\Omega}\left(\frac{1}{2}\rho_R|\vu_R |^2+ \mathcal{E}^{\rho_\infty}(\rho_R )\right)(\tau,\cdot) \dy + \int_{\Omega} \d\mathfrak{E}_R(\tau) \\
&\quad +\int^\tau_0\int_{\Omega}\left(\mu|\nabla_x\vu_R |^2+\delta|\partial_z\vu_R|^2 + (\mu+\lambda)|\Div\vu_R |^2\right) \dydt \\
&\quad \leq \int_{\Omega}\left(\frac{1}{2}\frac{|(\rho\vu)_{0,R}|^2}{\rho_{0,R}} +\mathcal{E}^{\rho_\infty}(\rho_{0,R} )\right) \dy + \int_0^\tau\int_{\Omega} \rho_R\vf\cdot\vu_R \dydt
\end{split}
\end{equation}
for almost all $\tau \in (0,T)$.

\subsection{The limit $R \to \infty$}

We start with the a priori estimates arising from \eqref{eq:EIvar31}. The source term on the right hand side can be estimated similarly as in \eqref{eq:RHSsplit}-\eqref{eq:RHSsplit2}. Here we use essential and residual parts with respect to $\rho_R$, i.e. $[h]_{ess,\rho_R}$ and $[h]_{res,\rho_R}$. Note in particular that for all $R \in \N$, $\vu_R \in D^{1,2}_0(\Omega)$, where $D^{1,2}_0(\Omega)$ denotes the homogeneous Sobolev space on the unbounded domain $\Omega$, therefore we still have the inequality
\begin{equation}
    \|\vu_R\|_{L^6(\Omega)} \leq C \|\nabla_y\vu_R\|_{L^2(\Omega)}
\end{equation}
with a constant independent of $R \in \N$ (see \cite[Section 1.3.6.4]{n}). The arising norms of $\vf$ can be easily bounded using the assumption $\vf \in L^\infty((0,T)\times\Omega) \cap L^\infty(0,T;L^1(\Omega))$.

Thus, we are able to deduce the following estimates
\begin{align}
    & \|1_{ess,\rho_R}(\rho_R-\rhoin)\|_{L^\infty(0,T;L^2(\Omega))} \leq C, \label{eq:essential} \\
    & \|1_{res,\rho_R}(\rho_R-\rhoin)\|_{L^\infty(0,T;L^\gamma(\Omega))} \leq C, \label{eq:residual} \\
    &\|\sqrt{\rho_R} \vu_R\|_{L^\infty(0,T;L^2(\Omega))} \leq C, \label{eq:sqrrhoubound}\\
    &\|\nabla_y\vu_R\|_{L^2(0,T;L^2(\Omega))} \leq C, \\
    &\|\mathfrak{E}_R\|_{L^\infty(0,T;\mathcal{M}(\overline{\Omega}))} \leq C.
\end{align}
Note that since $[1]_{res,\rho_R}|\rho_R-\rhoin| \geq \frac{\rhoin}{2}$, the bound \eqref{eq:residual} implies also
\begin{equation}\label{eq:resmeasure}
    \|[1]_{res,\rho_R}\|_{L^\infty(0,T;L^1(\Omega))} \leq C
\end{equation}
as well as
\begin{equation}
    \|[\rho_R]_{res,\rho_R}\|_{L^\infty(0,T;L^\gamma(\Omega))} \leq C.
\end{equation}
We use this information together with \eqref{eq:essential} and \eqref{eq:residual} to conclude that for any $m \in \N$ and any $R \geq m$ we have
\begin{equation}\label{eq:rhorbdd}
    \|\rho_R\|_{L^\infty(0,T;L^\gamma(\Omega_m))} \leq C(m).
\end{equation}
Similarly, using estimates
\begin{equation*}
\big{(}[\rho_R]_{res,\rho_R}\mathbf u_R\big{)}^{\frac{2\gamma}{\gamma+1}}\leq C\big{(}\rho_R|\mathbf u_R|^2\big{)}^{\frac{\gamma}{\gamma+1}}
([\rho_R]_{res,\rho_R}^\gamma)^{\frac{1}{\gamma+1}},\hspace{15pt}
\big{(}[\rho_R]_{ess,\rho_R}\mathbf u_R\big{)}^{2}\leq C\rho_R|\mathbf u_R|^2,
\end{equation*}
we end up with
\begin{equation}\label{4.21}
\|\rho_R\mathbf u_{R}\|_{L^\infty(0,T; L^\frac{2\gamma}{\gamma+1}(\Omega_m))}\leq C(m).
\end{equation}

% Next, we extend $\rho_{R}$ (defined so far a.e. on $\Omega_R$) to $\Omega$ by setting $\rho_{R}=\rho_\infty$ in $\Omega\backslash\Omega_R$. Then we extend $\mathbf u_{R}$ to $\mathbb{R}^3$ by setting $\mathbf u_{R}=0$. For a function $h$ defined on $\Omega_R$, we introduce its residual and essential parts with respect to $\rho$ by setting
% \begin{align}\label{4.40}
% %[h]_{ess}=h1_{\frac{\rho_\infty}{2}\leq\rho_R\leq2\rho_\infty},\hspace{5pt}[h]_{res}=h-[h]_{ess},\\
% [h]_{ess,\rho_R}=h1_{\frac{\rho_\infty}{2}\leq\rho_R\leq2\rho_\infty},\hspace{5pt}[h]_{ess,\rho_R}=h-[h]_{ess,\rho_R}.
% \end{align}
% {\color{red}
% \begin{align}\label{4.42}
% &ess\sup_{(0,T)}\int_{\Omega}\rho^{\gamma}_R[1]_{res,\rho_R}dx\leq C,\\
% &ess\sup_{(0,T)}\int_{\Omega}[1]_{res,\rho_R}dx\leq C,\\
% &ess\sup_{(0,T)}\int_{\Omega}(\rho_R-\rho_\infty)^2[1]_{ess,\rho_R}dx\leq C,\hspace{5pt}
% ess\sup_{(0,T)}\int_{\Omega}(\rho_R-\rho_\infty)^\gamma[1]_{res,\rho_R}dx\leq C.
% \end{align} }

We need the following version of the Poincar\' e inequality.
\begin{lemma}(\cite{j} Lemma 3.1)
For any $M>0$ there exists $c(M)>0$ such that
\begin{align*}
&\forall w\in W^{1,2}(\mathbb{R}^3),\hspace{2pt} \forall V\subset\mathbb{R}^3, \hspace{2pt} |V|\leq M,\\
&\|w\|^2_{L^2(\mathbb{R}^3)}\leq c(M)\left(\|\nabla w\|^2_{L^2(\mathbb{R}^3)}+\int_{\mathbb{R}^3\setminus V}w^2 \dy\right).
\end{align*}
\end{lemma}

We apply this Lemma for $w = \vu_{R}$ (prolonged by zero from $\Omega$ to $\R^3$), $V = \Omega_{res}(t)$, where
\begin{align*}
    &\Omega_{ess}(t) = \left\{x \in \Omega: \rho_R(t,x) \in \left(\frac{\rhoin}{2},2\rhoin\right) \right\}, \\
    &\Omega_{res}(t) = \Omega \setminus \Omega_{ess}(t).
\end{align*}
Note that \eqref{eq:resmeasure} provides the existence of $M$ independent of $t \in (0,T)$ bounding the measure of the set $\Omega_{res}(t)$. This way we obtain
\begin{equation}
    \|\vu_R(t,\cdot)\|_{L^2(\Omega)}^2 \leq C + \int_{\Omega_{ess}(t)} |\vu_R(t,\cdot)|^2 \dy \leq C\left(1 + \int_{\Omega_{ess}(t)} \rho_R(t,\cdot)|\vu_R(t,\cdot)|^2 \dy\right) \leq C.
\end{equation}
Hence we are able to deduce
\begin{align*}
\|\mathbf u_R\|_{L^2(0,T;W^{1,2}(\Omega))}\leq C.
\end{align*}

The bounds above together with Lemma \ref{l2.3} give rise to the following convergences
\begin{align}
&\rho_R \to \rho \quad \text{weakly}^* \text{ in } L^\infty(0,T;L^\gamma(\Omega_m)) \quad \text{for all } m \in \N, \label{eq:rhoconv} \\
&\vu_R \to \vu \quad \text{weakly} \text{ in } L^2(0,T;W^{1,2}(\Omega)), \\
&\rho_R\vu_R \to \vm \quad \text{weakly}^* \text{ in } L^\infty(0,T;L^\frac{2\gamma}{\gamma+1}(\Omega_m)) \quad \text{for all } m \in \N, \label{eq:rhouconv}\\
&\mathfrak{E}_R \to \mathfrak{E}_1 \quad \text{weakly}^* \text{ in } L^\infty(0,T;\mathcal{M}(\overline{\Omega})), \\
&\mathfrak{R}_R \to \mathfrak{R}_1 \quad \text{weakly}^* \text{ in } L^\infty(0,T;\mathcal{M}_{3\times 3}(\overline{\Omega})).
\end{align}
Moreover, it is not difficult to show using the continuity equation \eqref{eq:CEweak_lim21} that for any $m\in\mathbb{N}$, the sequence $\rho_{R}$ is uniformly continuous in $W^{-1,\frac{2\gamma}{\gamma+1}}(\Omega_m)$. Since $\rho_R$ belongs to $C_{weak}([0,T];L^\gamma(\Omega_m))$, and is uniformly bounded in $L^\gamma(\Omega_m)$, we end up with
\begin{equation}\label{4.23}
\rho_{R}\rightarrow\rho \quad \text{ in } C_{weak}([0,T];L^\gamma(\Omega_m)).
\end{equation}
Similarly, we use the momentum equation \eqref{eq:MEweak_lim21} to deduce that for all $m \in \N$ and all $\vphi \in C^\infty(\Omega_m)$ it holds
\begin{equation*}
    \left|\int_{\Omega_m} (\rho_R\vu_R)(\tau,\cdot)\cdot \vphi \dy - \int_{\Omega_m} (\rho_R\vu_R)(s,\cdot)\cdot \vphi \dy\right| \leq C\left|\tau-s\right|^{\frac 12},
\end{equation*}
which together with \eqref{eq:rhouconv} gives rise to
\begin{equation}
\rho_R\vu_R \to \vm \quad \text{ in } C_{weak}([0,T];L^\frac{2\gamma}{\gamma+1}(\Omega_m)) \quad \text{for all } m \in \N. \label{eq:rhouconv2}
\end{equation}
Next, we identify the limit of the momentum using Lemma \ref{l:rhou_iden} on domains $Q_m = (0,T) \times \Omega_m$ for any $m \in \N$ similarly as in Section \ref{s:limit_n}. We end up with
\begin{equation}
    \vm = \rho\vu \quad \text{ a.e. in } (0,T)\times \Omega_m \quad \text{ for any } m\in\N. \label{eq:rhoum}
\end{equation}
Now we are able to pass to the limit in \eqref{eq:CEweak_lim21}. For every fixed test function $\varphi \in C^1_c([0,T]\times\overline{\Omega})$ we find $m \in \N$ sufficiently large such that $\varphi = 0$ in $[0,T] \times \left(\Omega \setminus \Omega_m\right)$ and use \eqref{4.23}, \eqref{eq:rhouconv} and \eqref{eq:rhoum} to conclude that
\begin{equation}\label{eq:CEweak_lim3}
\int_{\Omega}(\rho\varphi)(\tau,y) \dy - \int_{\Omega}\rho_{0}(y)\varphi(0,y) \dy = \int^\tau_0\int_{\Omega}\left(\rho\partial_t\varphi + \rho\vu\cdot\Grad\varphi\right) \dydt,
\end{equation}
holds for all $\varphi\in C^1_c([0,T]\times\overline{\Omega})$ and almost all $\tau \in [0,T]$.

We proceed similarly in the limit of the momentum equation \eqref{eq:MEweak_lim21}. Again, for every test function $\vphi \in C^1([0,T];C^1_c(\Omega))$ we find $m \in \N$ sufficiently large such that $\vphi = 0$ in $[0,T] \times \left(\Omega \setminus \Omega_m\right)$. The bound \eqref{eq:rhorbdd} can be reformulated as
\begin{equation}
    \|p(\rho_R)\|_{L^\infty(0,T;\mathcal{M}(\Omega_m))} \leq C(m),
\end{equation}
so using Lemma \ref{l2.3} we get that
\begin{equation}
    p(\rho_R) \to \overline{p(\rho)} \quad \text{ weakly}^* \text{ in } L^\infty(0,T;\mathcal{M}(\Omega_m)) \text{ for any } m \in \N.
\end{equation}
The convexity of the pressure implies that
\begin{equation*}
    \mathfrak{R}_3 = (\overline{p(\rho)} - p(\rho))\mathbb{I} \geq 0.
\end{equation*}
Since the sequence $\rho_R|\vu_R|^2$, which can be rewritten as $\frac{|\vm_R|^2}{\rho_R}$ when $\rho_R > 0$, is uniformly bounded in $L^\infty(0,T;L^1(\Omega))$, we can similarly as in Section \ref{s:limit_n} conclude that
\begin{equation*}
    \rho_R\vu_R \otimes \vu_R = \frac{\vm_R \otimes \vm_R}{\rho_R} = E[\rho_R, \vm_R]\to \overline{E[\vr,\vm]} \quad \text{ weakly}^* \text{ in } L^\infty(0,T;\mathcal{M}_{3\times 3}(\overline{\Omega}))
\end{equation*}
with
\begin{equation*}
    \mathfrak{R}_2 := \overline{E[\vr,\vm]} - E[\vr,\vm] \geq 0,
\end{equation*}
due to the fact that $E_\xi[\rho,\vm]$ being a lower semicontinuous convex function for any $\xi \in \R^3$. Defining
\begin{equation*}
    \mathfrak{R} = \mathfrak{R}_1 + \mathfrak{R}_2 + \mathfrak{R}_3 \geq 0,
\end{equation*}
we may pass to the limit in \eqref{eq:MEweak_lim21} and recover
\begin{equation}\label{eq:MEweak_lim3}
\begin{split}
&\int_{\Omega}(\rho \vu \cdot\vphi)(\tau,y)\dy - \int_{\Omega}(\rho\vu)_{0}(y) \cdot\vphi(0,y) \dy  = \int^\tau_0\int_{\Omega} \Grad \vphi : \d\mathfrak{R}(t)\dt\\
& \quad + \int^\tau_0\int_{\Omega}\left(\rho \vu\cdot \partial_t\vphi + \rho \vu \otimes\vu :\Grad\vphi
+p(\rho)\Div\vphi\right)\dydt    \\
&\quad -\int^\tau_0\int_{\Omega}\left(\mu\nabla_x \vu :\nabla_x\vphi + \delta\partial_z\vu \cdot \partial_z\vphi
+(\mu+\lambda)\Div\vu\, \Div\vphi+\rho\vf\cdot\vphi\right) \dydt
\end{split}
\end{equation}
for all $\vphi\in C^1([0,T];C^1_c(\Omega))$ and almost all $\tau \in [0,T]$.

Finally, our aim is to pass to the limit in the energy inequality \eqref{eq:EIvar31}. In fact, we are going to show that it holds
\begin{equation}\label{eq:EIvar4}
\begin{split}
&\int_{\Omega_m}\left(\frac{1}{2}\rho|\vu |^2+ \mathcal{E}^{\rho_\infty}(\rho )\right)(\tau,\cdot) \dy + \int_{\Omega_m} \d\mathfrak{E}(\tau) \\
&\quad +\int^\tau_0\int_{\Omega_m}\left(\mu|\nabla_x\vu |^2+\delta|\partial_z\vu|^2 + (\mu+\lambda)|\Div\vu |^2\right) \dydt \\
&\quad \leq \int_{\Omega}\left(\frac{1}{2}\frac{|(\rho\vu)_{0}|^2}{\rho_{0}} +\mathcal{E}^{\rho_\infty}(\rho_{0} )\right) \dy + \int_0^\tau\int_{\Omega} \rho\vf\cdot\vu \dydt
\end{split}
\end{equation}
for almost all $\tau \in (0,T)$ and all sufficiently large $m \in \N$ with some $\mathfrak{E} \in L^\infty(0,T;\mathcal{M}^+(\overline{\Omega}))$. Then, since the right hand side is fixed, we obtain the validity of the energy inequality in the form \eqref{2.3}.
%{\color{blue} Ondrej: I have to say that I do not know how to deduce this unless $\vf$ has compact support. The arguments after this are therefore for the case that $\supp \ \vf(t,\cdot) \subset \Omega_N$ for all $t \in (0,T)$ for some $N \in \N$ and $m \geq N$.}

Let us start with \eqref{eq:EIvar31}. Since all terms on the left hand side of this inequality are non-negative, restricting their integrals from $\Omega$ to $\Omega_m$ can not increase the value of these integrals, so we get
\begin{equation}\label{eq:EIvar40}
\begin{split}
&\int_{\Omega_m}\left(\frac{1}{2}\rho_R|\vu_R |^2+ \mathcal{E}^{\rho_\infty}(\rho_R )\right)(\tau,\cdot) \dy + \int_{\Omega_m} \d\mathfrak{E}_R(\tau) \\
&\quad +\int^\tau_0\int_{\Omega_m}\left(\mu|\nabla_x\vu_R |^2+\delta|\partial_z\vu_R|^2 + (\mu+\lambda)|\Div\vu_R |^2\right) \dydt \\
&\quad \leq \int_{\Omega}\left(\frac{1}{2}\frac{|(\rho\vu)_{0,R}|^2}{\rho_{0,R}} +\mathcal{E}^{\rho_\infty}(\rho_{0,R} )\right) \dy + \int_0^\tau\int_{\Omega} \rho_R\vf\cdot\vu_R \dydt
\end{split}
\end{equation}
for almost all $\tau \in (0,T)$ and all $m \in \N$. Now, we pass to the limit in \eqref{eq:EIvar40}. The first integral on the right hand side converges to its natural counterpart. In order to handle the second one we show that for any $\varepsilon > 0$ it holds
\begin{equation*}
    \lim_{R \to \infty} \int_0^\tau \int_{\Omega} |\rho_R\vf\cdot\vu_R - \vr\vf\cdot\vu| \dydt \leq \varepsilon.
\end{equation*}
Fixing $\varepsilon$ we take $N \in \N$ sufficiently large and split
\begin{equation*}
    \int_0^\tau \int_{\Omega} |\rho_R\vf\cdot\vu_R - \vr\vf\cdot\vu| \dydt = \int_0^\tau \int_{\Omega_N} |(\rho_R\vu_R - \vr\vu)\cdot\vf| \dydt + \int_0^\tau \int_{\Omega \setminus \Omega_N} |(\rho_R\vu_R - \vr\vu)\cdot\vf| \dydt.
\end{equation*}
The first integral converges to zero as $R \to \infty$ due to \eqref{eq:rhouconv2} and \eqref{eq:rhoum}. We claim, that the second integral can be made arbitrarily small by taking $N$ large enough. Indeed, this surely holds for $\int_0^\tau\int_{\Omega \setminus\Omega_N} |\rho\vu\cdot\vf| \dydt$ as $\rho\vu\cdot\vf \in L^\infty(0,T;L^{\frac{2\gamma}{\gamma-1}}(\Omega))$ and we can estimate
\begin{align*}
    &\int_0^\tau \int_{\Omega \setminus \Omega_N} |\rho_R\vu_R\cdot\vf| \dy \leq C\|\sqrt{\rho_R}\vu_R\|_{L^\infty(0,T;L^2(\Omega))}\left(\int_0^\tau\int_{\Omega \setminus \Omega_N}\rho_R|\vf|^2\dydt\right)^{\frac 12} \\
    &\quad \leq C\left(\int_0^\tau\int_{\Omega \setminus \Omega_N}|\rho_R-\rhoin||\vf|^2\dydt\right)^{\frac 12} + C\|\vf\|_{L^2((0,T)\times(\Omega\setminus\Omega_N))} \\
    &\quad \leq C\left(\int_0^\tau\int_{\Omega \setminus \Omega_N}[1]_{res,\rho_R}|\rho_R-\rhoin||\vf|^2\dydt\right)^{\frac 12} + C\|\vf\|_{L^2((0,T)\times(\Omega\setminus\Omega_N))} \\
    &\quad \leq C\|[1]_{res,\rho_R}(\rho_R-\rhoin)\|_{L^\infty(0,T;L^\gamma(\Omega))}\left(\int_0^\tau\int_{\Omega \setminus \Omega_N}|\vf|^{\frac{2\gamma}{\gamma-1}}\dydt\right)^{\frac{\gamma-1}{2\gamma}} + C\|\vf\|_{L^2((0,T)\times(\Omega\setminus\Omega_N))} \\
    &\quad \leq C\|\vf\|_{L^{\frac{2\gamma}{\gamma-1}}((0,T)\times(\Omega\setminus\Omega_N))} + C\|\vf\|_{L^2((0,T)\times(\Omega\setminus\Omega_N))},
\end{align*}
where we used \eqref{eq:residual} and \eqref{eq:sqrrhoubound}. The arising terms can be made arbitrarily small by taking $N$ sufficiently large.

On the left hand side of \eqref{eq:EIvar40} we use the weak lower semicontinuity of convex functions (see Lemma \ref{l:convex}) to deduce that
\begin{equation*}
\begin{split}
    &\int^\tau_0\int_{\Omega_m}\left(\mu|\nabla_x\vu |^2+\delta|\partial_z\vu|^2 + (\mu+\lambda)|\Div\vu |^2\right) \dydt \\
    & \quad \leq \liminf_{R \to \infty} \int^\tau_0\int_{\Omega_m}\left(\mu|\nabla_x\vu_R |^2+\delta|\partial_z\vu_R|^2 + (\mu+\lambda)|\Div\vu_R |^2\right) \dydt
\end{split}
\end{equation*}
and similarly Lemma \ref{l2.4} to deduce
\begin{equation*}
    \int_{\Omega_m} \d \mathfrak{E}_1(\tau) \leq \liminf_{R \to \infty} \int_{\Omega_m} \d\mathfrak{E}_R(\tau).
\end{equation*}
Finally, the sequences $\rho_R|\vu_R|^2$ and $\mathcal{E}^{\rhoin}(\rho_R)$ are bounded in $L^\infty(0,T;L^1(\Omega))$ and thus the following convergences hold
\begin{equation*}
    \rho_R|\vu_R|^2 = \frac{|\vm_R|^2}{\rho_R} = E_0[\rho_R, \vm_R]\to \overline{E_0[\vr,\vm]} \quad \text{ weakly}^* \text{ in } L^\infty(0,T;\mathcal{M}(\overline{\Omega})),
\end{equation*}
and
\begin{equation*}
    \mathcal{E}^{\rho_\infty}(\rho_R) \to \overline{\mathcal{E}^{\rho_\infty}(\rho)} \quad \text{ weakly}^* \text{ in } L^\infty(0,T;\mathcal{M}(\overline{\Omega})),
\end{equation*}
with
\begin{equation*}
    \mathfrak{E}_2 = \overline{E_0[\vr,\vm]} - E_0[\vr,\vm] \geq 0,
\end{equation*}
and
\begin{equation*}
    \mathfrak{E}_3 = \overline{\mathcal{E}^{\rho_\infty}(\rho)} - \mathcal{E}^{\rho_\infty}(\rho) \geq 0,
\end{equation*}
whence
\begin{equation*}
    \mathfrak{E} = \mathfrak{E}_1 + \mathfrak{E}_2 + \mathfrak{E}_3 \geq 0.
\end{equation*}
This enables us to conclude \eqref{eq:EIvar4} and therefore also \eqref{2.3}. The relation \eqref{eq:REineq1} is preserved due to the same reasons as in Section \ref{s:limit_n}. Note that the linear term in $\mathcal{E}^{\rhoin}$ has a proper limit on $\Omega_m$ due to \eqref{4.23} and therefore it does not contribute to the measure $\mathfrak{E}$ at all. This finishes the proof of Theorem \ref{t3.1}.

\section{Proof of Theorem \ref{t:REI}}\label{s:5}

The proof of the relative entropy inequality follows standard steps. First, observe that $\vU$ is a proper test function in the momentum equation \eqref{2.2}, which then yields
\begin{equation}\label{eq:REI_U}
\begin{split}
&\int_{\Omega}(\rho\mathbf u\cdot\vU)(\tau,y) \dy = 
\int_{\Omega}(\rho\mathbf u)_0(y)\cdot\vU(0,y) \dy + \int^\tau_0\int_\Omega\Grad\vU:\d\mathfrak{R}(t)\dt \\
&\qquad + \int^\tau_0\int_{\Omega}\left(\rho\mathbf u\cdot\partial_t\vU + \rho\mathbf{u}\otimes\mathbf{u}:\Grad\vU
 + p(\rho)\Div\vU\right)\dydt \\
&\qquad - \int^\tau_0\int_{\Omega} \left(\mu\nabla_x\mathbf u:\nabla_x\vU + \delta\partial_z\mathbf u\cdot \partial_z\vU + (\mu+\lambda)\Div\mathbf u\ \Div\vU
+\rho \mathbf{f}\cdot\vU\right)\dydt.
\end{split}
\end{equation}
Similarly, $\frac12|\vU|^2$ and $P'(r)-P'(\rhoin)$ are suitable test functions in the continuity equation \eqref{2.1} and therefore we obtain the following two identities
\begin{equation}\label{eq:REI_U2}
\int_{\Omega}\frac12 \rho|\vU|^2(\tau,y) \dy = \int_{\Omega}\frac12 \rho_0(y)|\vU(0,y)|^2 \dy + \int^\tau_0\int_{\Omega}\left(\rho\vU\cdot\partial_t\vU + \rho(\vU\otimes\mathbf{u}): \Grad\vU\right) \dydt,
\end{equation}
and
\begin{equation}\label{eq:REI_Pprime}
\begin{split}
&\int_{\Omega}\rho(P'(r)-P'(\rhoin))(\tau,y) \dy = \int_{\Omega}\rho_0(y)(P'(r(0,y))-P'(\rhoin)) \dy \\
&\quad + \int^\tau_0\int_{\Omega}\left(\rho\partial_t P'(r) + \rho\mathbf{u}\cdot \Grad P'(r)\right) \dydt.
\end{split}
\end{equation}
Finally, we observe that it holds
\begin{equation}\label{eq:REI_press}
\begin{split}
&\int_{\Omega}(p(r)-p(\rhoin))(\tau,y) \dy = \int_{\Omega}(p(r(0,y))-p(\rhoin)) \dy + \int^\tau_0\int_{\Omega} \partial_t p(r) \dydt \\
&\quad = \int_{\Omega}(p(r(0,y))-p(\rhoin)) \dy + \int^\tau_0\int_{\Omega} r\partial_t P'(r) \dydt,
\end{split}
\end{equation}
where we used \eqref{eq:pPrel}. Now we sum \eqref{2.3} together with \eqref{eq:REI_U2} and \eqref{eq:REI_press} and subtract \eqref{eq:REI_U} and \eqref{eq:REI_Pprime} and use also \eqref{eq:pPrel}. Straightforward calculations show that all terms containing $\rhoin$ cancel and we recover \eqref{eq:REI} with right hand side given as in \eqref{eq:Remainder}. The proof of Theorem \ref{t:REI} is finished.

\section{Proof of Theorem \ref{t:WSU}}\label{s:6}

The main idea is to use the relative entropy inequality \eqref{eq:REI} with test functions $(r,\vU) = (\tilde \rho,\tilde \vu)$, where $(\tilde\rho,\tilde\vu)$ is the strong solution to the problem \eqref{1a}-\eqref{eq:IC}. In order to to this however, we first need to argue that this is admissible, since the strong solution is not compactly supported and hence it cannot be directly chosen as a proper test function. The class of test functions in \eqref{eq:REI} can be however significantly enlarged using a standard density argument. Therefore, we conclude, that $(r,\vU)$ is a proper test function in \eqref{eq:REI} as long as all integrals in the inequality are finite and the boundary behavior $\vU|_{\partial\Omega} = 0$, $(r,\vU) \to (\rhoin,0)$ as $y \to \infty$ is satisfied. It is not difficult to check that for couples of functions in the regularity class \eqref{eq:WSU_reg} this indeed holds and therefore such couple is indeed a proper test function.

Using standard arguments we manipulate terms in the remainder $\mathcal{R}(\rho,\vu,\mathfrak{R},\vrt,\vut)$ defined in \eqref{eq:Remainder} as follows. First,
\begin{equation}\label{5.6}
\begin{split}
&\int_\Omega\rho(\partial_t\vut+\mathbf u\cdot\nabla_y\vut)\cdot(\vut-\mathbf u) \dy \\
&\quad = \int_\Omega\rho(\partial_t\vut+\vut\cdot\nabla_y\vut)\cdot(\vut-\mathbf u)\dy - \int_\Omega\rho(\mathbf u-\vut)\otimes(\mathbf u-\vut) :\nabla_y\vut\dy.
\end{split}
\end{equation}

Next, by virtue of the strong formulation of the system (1.2), and the fact that $(\vrt,\vut)$ is the strong solution, we conclude
\begin{equation*}
\vut_t+\vut\cdot\nabla_y\vut=-\frac{1}{\vrt}\nabla_y p(\vrt)+ \frac{1}{\vrt}\left(\mu\Delta_x\vut+\delta\partial_{zz}\vut+(\mu+\lambda)\nabla_y\Div\vut\right).
\end{equation*}
Therefore, we get
\begin{equation}\label{5.8}
\begin{split}
&\int_\Omega\rho(\partial_t\vut+\vut\cdot\nabla_y\vut)\cdot(\vut-\mathbf u)\dy = - \int_\Omega \frac{\rho}{\vrt}\nabla_y p(\vrt)\cdot(\vut-\vu)\dy \\
&\quad + \int_\Omega \frac{\rho-\vrt}{\vrt}(\mu\Delta_x\vut+\delta\partial_{zz}\vut+(\mu+\lambda)\nabla_y\Div\vut)\cdot(\vut-\vu)\dy \\
&\quad + \int_\Omega (\mu\Delta_x\vut+\delta\partial_{zz}\vut+(\mu+\lambda)\nabla_y\Div\vut)\cdot(\vut-\vu)\dy,
\end{split}
\end{equation}
where the last integral can be integrated by parts to obtain
\begin{equation}
\begin{split}
&\int_\Omega (\mu\Delta_x\vut+\delta\partial_{zz}\vut+(\mu+\lambda)\nabla_y\Div\vut)\cdot(\vut-\vu)\dy \\
&\quad = - \int_\Omega (\mu\nabla_x\vut:\nabla_x(\vut-\vu)+\delta\partial_{z}\vut\cdot \partial_{z}(\vut-\vu)+(\mu+\lambda)\Div\vut\ \Div(\vut-\vu)\dy.
\end{split}
\end{equation}
Moving the above expressions to the left hand side of the relative entropy inequality \eqref{eq:REI} together with the integral on the first line of the right hand side of \eqref{eq:Remainder} and we obtain the following version of the inequality
\begin{equation}\label{eq:REI_step1}
\begin{split}
   &\int_{\Omega}\left(\frac 12 \rho |\vu-\vut|^2 + \mathcal{E}^{\vrt}(\rho)\right)(\tau,y) \dy + \int_\Omega \d\mathfrak{E}(\tau) \\
   &\quad + \int^\tau_0\int_\Omega\left(\mu|\nabla_x\mathbf u-\nabla_x\vut|^2+\delta|\partial_z\mathbf u - \partial_z\vut|^2+(\mu+\lambda)|\Div\mathbf u-\Div\vut|^2\right)\dydt \\
   &\quad \leq - \int_0^\tau \int_\Omega \nabla_y \vut:\d \mathfrak{R}(t) \dt -\int_0^\tau\int_\Omega \rho(\vu-\vut)\otimes(\vu-\vut):\nabla_y\vut \dydt \\
   &\quad + \int_0^\tau\int_\Omega \frac{\rho-\vrt}{\vrt}\left( \mu\Delta_x\vut + \delta\partial_{zz}\vut + (\mu+\lambda)\nabla_y\Div\vut \right)\cdot(\vut-\vu) \dydt \\
   &\quad - \int_0^\tau\int_\Omega\left( p(\rho)\Div\vut + (\rho-\vrt)\partial_t P'(\vrt) + \rho\vu\cdot\nabla_y P'(\vrt)+ \frac{\rho}{\vrt}\nabla_y p(\vrt)\cdot(\vut-\vu)\right) \dydt.  
\end{split}
\end{equation}
Next, we use \eqref{eq:pPrel} to simplify the last two terms on the right hand side of \eqref{eq:REI_step1}:
\begin{equation}
\begin{split}
    &\int_0^\tau\int_\Omega \left(\rho\vu\cdot\nabla_y P'(\vrt)+ \frac{\rho}{\vrt}\nabla_y p(\vrt)\cdot(\vut-\vu)\right) \dydt \\
    &\quad = \int_0^\tau\int_\Omega\left(\rho\vu\cdot\nabla_y P'(\vrt) + \rho\nabla_y P'(\vrt)\cdot(\vut-\vu)\right)\dydt = \int_0^\tau\int_\Omega \rho\vut\cdot\nabla_y P'(\vrt) \dydt.
\end{split}
\end{equation}
Finally, we use the continuity equation for the strong solution $(\vrt,\vut)$ together with the fact that
\begin{equation*}
    p'(\rho) = P''(\rho)\rho,
\end{equation*}
which follows from \eqref{eq:pPrel}, to further rewrite the last line of \eqref{eq:REI_step1} to
\begin{equation}
\begin{split}
    &\int_0^\tau\int_\Omega\left( p(\rho)\Div\vut + (\rho-\vrt)\partial_t P'(\vrt) + \rho\vut\cdot\nabla_y P'(\vrt)\right) \dydt \\
    &\quad = \int_0^\tau\int_\Omega\left(p(\rho)\Div\vut +  (\rho-\vrt)P''(\vrt)\partial_t \vrt + \rho P''(\vrt)\vut\cdot\nabla_y\vrt \right) \dydt \\
    &\quad = \int_0^\tau\int_\Omega\left(p(\rho)\Div\vut -  (\rho-\vrt)P''(\vrt) \vrt \Div\vut + \vrt P''(\vrt)\vut\cdot\nabla_y\vrt \right) \dydt \\
    &\quad = \int_0^\tau\int_\Omega\left(p(\rho)\Div\vut -  (\rho-\vrt)p'(\vrt) \Div\vut + \vut\cdot\nabla_y p(\vrt) \right) \dydt \\
    &\quad = \int_0^\tau\int_\Omega \Div\vut \left(p(\rho) - p(\vrt) - p'(\vrt)(\rho-\vrt) \right) \dydt,
\end{split}
\end{equation}
where we used also integration by parts on the last line. Therefore we obtain
\begin{equation}\label{eq:REI_step2}
\begin{split}
   &\int_{\Omega}\left(\frac 12 \rho |\vu-\vut|^2 + \mathcal{E}^{\vrt}(\rho)\right)(\tau,y) \dy + \int_\Omega \d\mathfrak{E}(\tau) \\
   &\quad + \int^\tau_0\int_\Omega\left(\mu|\nabla_x\mathbf u-\nabla_x\vut|^2+\delta|\partial_z\mathbf u - \partial_z\vut|^2+(\mu+\lambda)|\Div\mathbf u-\Div\vut|^2\right)\dydt \\
   &\quad \leq - \int_0^\tau \int_\Omega \nabla_y \vut:\d \mathfrak{R}(t) \dt -\int_0^\tau\int_\Omega \rho(\vu-\vut)\otimes(\vu-\vut):\nabla_y\vut \dydt \\
   &\quad + \int_0^\tau\int_\Omega \frac{\rho-\vrt}{\vrt}\left( \mu\Delta_x\vut + \delta\partial_{zz}\vut + (\mu+\lambda)\nabla_y\Div\vut \right)\cdot(\vut-\vu) \dydt \\
   &\quad - \int_0^\tau\int_\Omega \Div\vut \left(p(\rho) - p(\vrt) - p'(\vrt)(\rho-\vrt) \right) \dydt. 
\end{split}
\end{equation}
We are now ready to estimate the four integrals on the right hand side of \eqref{eq:REI_step2} by means of terms on the left hand side. In order to shorten notation, let us denote
\begin{equation}
    \mathcal{E}_{rel}(\tau) := \int_{\Omega}\left(\frac 12 \rho |\vu-\vut|^2 + \mathcal{E}^{\vrt}(\rho)\right)(\tau,y) \dy + \int_\Omega \d\mathfrak{E}(\tau).
\end{equation}
Then, using \eqref{2.5} and the fact that $\mathfrak{R}$ is a positive semi-definite matrix, we get
\begin{equation}\label{eq:RHS1}
    \left|\int_0^\tau \int_\Omega \nabla_y \vut:\d \mathfrak{R}(t) \dt\right| \leq C\int_0^\tau\|\nabla_y\vut\|_{L^\infty(\Omega)} \mathcal{E}_{rel}(t)\dt.
\end{equation}
Similarly, we have
\begin{equation}\label{eq:RHS2}
    \left|\int_0^\tau\int_\Omega \rho(\vu-\vut)\otimes(\vu-\vut):\nabla_y\vut \dydt\right| \leq C\int_0^\tau\|\nabla_y\vut\|_{L^\infty(\Omega)} \mathcal{E}_{rel}(t)\dt
\end{equation}
and since it is easy to observe that
\begin{equation*}
    p(\rho) - p(\vrt) - p'(\vrt)(\rho-\vrt) = (\gamma-1)\mathcal{E}^{\vrt}(\rho),
\end{equation*}
we also get
\begin{equation}\label{eq:RHS3}
    \left|\int_0^\tau\int_\Omega \Div\vut \left(p(\rho) - p(\vrt) - p'(\vrt)(\rho-\vrt) \right) \dydt\right| \leq C\int_0^\tau \|\nabla_y\vut\|_{L^\infty(\Omega)}\mathcal{E}_{rel}(t)\dt. 
\end{equation}
It remains to handle the third integral on the right hand side of \eqref{eq:REI_step2}, where again in order to shorten the notation we denote
\begin{equation*}
    \mathbf{d}(\nabla_y^2\vut) :=  \mu\Delta_x\vut + \delta\partial_{zz}\vut + (\mu+\lambda)\nabla_y\Div\vut.
\end{equation*}
We split the integral into three parts
\begin{equation*}
\begin{split}
    &\int_0^\tau\int_\Omega \frac{\rho-\vrt}{\vrt}\mathbf{d}(\nabla_y^2\vut)\cdot(\vut-\vu) \dydt = \int_{\rho < \frac{\vrt}{2}} \frac{\rho-\vrt}{\vrt}\mathbf{d}(\nabla_y^2\vut)\cdot(\vut-\vu) \dydt \\
    &\quad + \int_{\frac{\vrt}{2}\leq\rho\leq 2\vrt} \frac{\rho-\vrt}{\vrt}\mathbf{d}(\nabla_y^2\vut)\cdot(\vut-\vu) \dydt + \int_{\rho > 2\vrt} \frac{\rho-\vrt}{\vrt}\mathbf{d}(\nabla_y^2\vut)\cdot(\vut-\vu) \dydt = \sum_{i=1}^3 I_i.
\end{split}
\end{equation*}
Before continuing, we recall \eqref{eq:Pass}, which yields the following estimates
\begin{align}
    1 \leq C(\vrt) \mathcal{E}^{\vrt}(\rho) \quad &\text{ for } \rho < \frac{\vrt}{2} \label{6.12}\\
    (\rho-\vrt)^2 \leq C(\vrt) \mathcal{E}^{\vrt}(\rho) \quad &\text{ for } \frac{\vrt}{2} \leq \rho \leq 2\vrt \label{6.13}\\
    \rho^\gamma \leq C(\vrt) \mathcal{E}^{\vrt}(\rho) \quad &\text{ for } \rho > 2\vrt, \label{6.14}
\end{align}
where $C(\vrt)$ is a constant depending on $\inf \vrt$ and $\sup \vrt$. Then we have
\begin{equation}\label{eq:RHS41}
\begin{split}
    &|I_1| \leq C\int_0^\tau \|\vut-\vu\|_{L^6(\Omega)}\left\|\frac{\mathbf{d}(\nabla_y^2\vut)}{\vrt}\right\|_{L^3(\Omega)}\|1\|_{L^2(\{y:\rho(t,y)<\frac{\vrt(t,y)}{2}\})} \dt \\
    &\quad \leq \delta\int_0^\tau\|\vut-\vu\|_{L^6(\Omega)}^2\dt + C(\delta)\int_0^\tau\|\nabla_y^2\vut\|_{L^3(\Omega)}^2\mathcal{E}_{rel}(t)\dt \\
    &\quad \leq \delta_0\int^\tau_0\int_\Omega\left(\mu|\nabla_x\mathbf u-\nabla_x\vut|^2+\delta|\partial_z\mathbf u - \partial_z\vut|^2+(\mu+\lambda)|\Div\mathbf u-\Div\vut|^2\right)\dydt \\
    &\quad + C(\delta)\int_0^\tau\|\nabla_y^2\vut\|_{L^3(\Omega)}^2\mathcal{E}_{rel}(t)\dt,
\end{split}
\end{equation}
where we have used \eqref{6.12}, the Sobolev inequality, Corollary \ref{co:beta} and the fact that $\vrt$ is bounded from below by a positive constant. The first term can now be moved to the left hand side of \eqref{eq:REI_step2} by choosing $\delta$ sufficiently small.

With the second integral we proceed similarly.
\begin{equation}\label{eq:RHS42}
\begin{split}
    &|I_2| \leq C\int_0^\tau \|\vut-\vu\|_{L^6(\Omega)}\left\|\frac{\mathbf{d}(\nabla_y^2\vut)}{\vrt}\right\|_{L^3(\Omega)}\|\rho-\vrt\|_{L^2(\{y:\frac{\vrt(t,y)}{2}\leq\rho(t,y)\leq2\vrt(t,y)\})} \dt \\
    &\quad \leq \delta\int_0^\tau\|\vut-\vu\|_{L^6(\Omega)}^2\dt + C(\delta)\int_0^\tau\|\nabla_y^2\vut\|_{L^3(\Omega)}^2\mathcal{E}_{rel}(t)\dt \\
    &\quad \leq \delta_0\int^\tau_0\int_\Omega\left(\mu|\nabla_x\mathbf u-\nabla_x\vut|^2+\delta|\partial_z\mathbf u - \partial_z\vut|^2+(\mu+\lambda)|\Div\mathbf u-\Div\vut|^2\right)\dydt \\
    &\quad + C(\delta)\int_0^\tau\|\nabla_y^2\vut\|_{L^3(\Omega)}^2\mathcal{E}_{rel}(t)\dt,
\end{split}
\end{equation}
where we have used the same tools as above with \eqref{6.13} instead of \eqref{6.12}.

Finally, 
%in order to handle the third integral, we have to realize that we already know that 
% \begin{equation}
%     \mathcal{E}_{rel} \in L^\infty(0,T).
% \end{equation}
% This is a consequence of the fact that for a dissipative turbulent solution $(\rho,\vu)$ with Reynolds defect measure $\mathfrak{R}$ and a strong solution $(\vrt,\vut)$ belonging to the class \eqref{eq:WSU_reg}, it holds
% $$
% \int_0^T \mathcal{R}(\rho,\vu,\mathfrak{R},\vrt,\vut)(t)\dt \leq C.
% $$
we estimate the integral $I_3$ as follows.
\begin{equation}\label{eq:RHS43}
\begin{split}
    &|I_3| \leq \int_0^\tau\int_\Omega \sqrt{\rho}|\vu-\vut|\sqrt{\rho} \left|\frac{\mathbf{d}(\nabla_y^2\vut)}{\vrt}\right| 1_{\rho > 2\vrt} \dydt \\
    &\quad \leq C\int_0^\tau \|\rho(\vu-\vut)^2\|_{L^1(\Omega)}^{\frac 12}\|\rho 1_{\rho > 2\vrt}\|_{L^1(\Omega)}^{\frac 12} \|\nabla^2_y\vut\|_{L^\infty(\Omega)} \dt \\
    &\quad \leq C \int_0^\tau \mathcal{E}_{rel}(t) \|\nabla^2_y\vut\|_{L^\infty(\Omega)} \dt,
\end{split}
\end{equation}
where we have used also $\rho < C\rho^\gamma$ on the set $\{\rho > 2\vrt\}$ and \eqref{6.14}.

Summing up \eqref{eq:RHS1}, \eqref{eq:RHS2}, \eqref{eq:RHS3}, \eqref{eq:RHS41}, \eqref{eq:RHS42} and \eqref{eq:RHS43} we obtain from \eqref{eq:REI_step2}
\begin{equation}\label{eq:REI_step3}
\begin{split}
   &\mathcal{E}_{rel}(\tau) + \int^\tau_0\int_\Omega\left(\mu|\nabla_x(\mathbf u-\vut)|^2+\delta|\partial_z(\mathbf u -\vut)|^2+(\mu+\lambda)|\Div(\mathbf u-\vut)|^2\right)\dydt \\
   &\quad \leq C\int_0^\tau h(t)\mathcal{E}_{rel}(t)\dt,
\end{split}
\end{equation}
where 
$$
h(t) = \|\nabla_y\vut\|_{L^\infty(\Omega)} + \|\nabla_y^2\vut\|_{L^3(\Omega)}^2 + \|\nabla_y^2\vut\|_{L^\infty(\Omega)} \in L^1(0,T),
$$
due to the assumptions \eqref{eq:WSU_reg}. This however directly implies $\mathcal{E}_{rel} \equiv 0$ in $(0,T)$ using the Gronwall lemma. This in turn yields $\mathfrak{E} \equiv 0$ and also $\mathfrak{R} \equiv 0$ through \eqref{2.5}, $\rho = \vrt$ and finally $\vu = \vut$ in $(0,T)\times \Omega$. This completes the proof of Theorem \ref{t:WSU}.

\vskip 0.5cm
\noindent {\bf Acknowledgements}

\vskip 0.1cm

The research of  O. K. and \v{S}. N. has been supported by the Czech Science Foundation (GA\v CR) project 22-01591S.  Moreover,  O. K. and \v S. N.  has been supported by  Praemium Academiae of \v S. Ne\v casov\' a. The Institute of Mathematics, CAS is supported by RVO:67985840. T. T. is partially supported by NSFC No. 12371246 and Qing Lan Project of Jiangsu Province.

%%%%%%%%%%%%%%%%%%%%%%%%%%%%%%%%%%%%%%%%%%%%%%%%%%%%%%%%%

\end{document}